\documentclass[11pt,reqno,a4paper]{article}
\usepackage{amsmath}
\usepackage{amsbsy}
\usepackage{amsthm}
\usepackage{amssymb}
\usepackage{amscd}
\usepackage{accents}
\usepackage{float}
\usepackage{url}
\usepackage{tikz}
\usetikzlibrary{matrix,arrows,decorations.pathmorphing}
 \usepackage{relsize}
\usepackage{xcolor, mathrsfs}

\usepackage{enumerate}
\usepackage{comment}
\usepackage{hyperref}
\hypersetup{
pdftitle={},%
pdfauthor={},%
pdfsubject={},%
pdfkeywords={},%
colorlinks=true,%
linkcolor={black},%
linktoc={},
linktocpage={false},%
pageanchor={},
citecolor={black},
}

\usepackage[english]{babel}
\usepackage[utf8]{inputenc}
\usepackage[margin=2.41cm]{geometry}

\usepackage{hyperref}
\usepackage{url}

\usepackage{cite}

\usepackage{titlesec}
\titleformat{\section}{\large\bfseries\filcenter}{\thesection}{1em}{}
\titleformat{\subsection}{\bfseries}{\thesubsection}{1em}{}

\input{xy}
\xyoption{all}

\newtheorem{theorem}{Theorem}[section]

\newtheorem{lemma}[theorem]{Lemma}
\newtheorem{proposition}[theorem]{Proposition}
\theoremstyle{remark}

\theoremstyle{definition}
\newtheorem{remark}[theorem]{Remark}
\newtheorem{definition}[theorem]{Definition}
\newtheorem{example}[theorem]{Example}
\newtheorem{notation}[theorem]{Notation}

\numberwithin{equation}{section}
\allowdisplaybreaks[1]

\catcode`@=12

\renewcommand\thanks[1]{%
  \begingroup
  \renewcommand\thefootnote{}\footnote{#1}%
  \addtocounter{footnote}{-1}%
  \endgroup
}

\renewcommand{\tilde}{\widetilde}
\renewcommand{\epsilon}{{\varepsilon}}

\usepackage{bbm}

\newcommand{\ie}{\textit{i.e. }}
\newcommand{\cf}{\textit{cf. }}

\newcommand{\KK}{{\mathbb{K}}}
\newcommand{\RR}{{\mathbb{R}}}
\newcommand{\CC}{{\mathbb{C}}}
\newcommand{\n}{{\mathtt{n}}}
\renewcommand{\i}{\imath}

\newcommand{\Dir}{{\mathsf{D}}}
\newcommand{\G}{\mathsf{G}}
\newcommand{\Id}{{\operatorname{Id}}}

\newcommand{\tr}{\mathrm{tr}}

\newcommand{\E}{\mathsf{E}}

\renewcommand{\L}{\mathsf{L}}
\newcommand{\M}{\mathsf{M}}

\renewcommand{\H}{\mathsf{H}}

\renewcommand{\P}{\mathsf{P}}

\renewcommand{\S}{\mathbb{S}}
\newcommand{\T}{\mathsf{T}}
\newcommand{\U}{\mathsf{U}}
\newcommand{\V}{\mathsf{V}}

\newcommand{\bM}{{\partial\M}}
\newcommand{\bSigma}{{\partial\Sigma}}

\newcommand{\oS}{\mathsf{S}}
\newcommand{\oK}{{\mathsf{K}_\lambda}}

\newcommand{\fS}{\mathfrak{S}}

\newcommand{\oB}{\mathsf{B}}
\newcommand{\f}{{\mathfrak{f}}}

\newcommand{\h}{{\mathfrak{h}}}

\newcommand{\vol}{{\textnormal{vol}\,}}

\newcommand{\supp}{{\textnormal{supp\,}}}
\newcommand{\Spin}{\textnormal{Spin}}

\newcommand{\bracket}[2]{\langle  #1\,|\, #2  \rangle}
\newcommand{\fiber}[2]{\prec  #1\,|\, #2  \succ}

\newcommand{\scalar}[2]{(#1\, |\, #2)}

\newcommand{\black}{\color{black}}

\begin{document}

\begin{flushright}

\baselineskip=4pt

\end{flushright}

\begin{center}
\vspace{5mm}

{\Large\bf ON THE CAUCHY PROBLEM FOR FRIEDRICHS SYSTEMS \\[3MM] ON GLOBALLY HYPERBOLIC MANIFOLDS\\[4MM] WITH TIMELIKE BOUNDARY}

\vspace{5mm}

{\bf by}

\vspace{5mm}
{  \bf  Nicolas Ginoux}\\[1mm]
\noindent  {\it Universit\'e de Lorraine, CNRS, IECL, F-57000 Metz, France}\\[1mm]
email: \ {\tt  nicolas.ginoux@univ-lorraine.fr}\\

\vspace{5mm}

{  \bf  Simone Murro}\\[1mm]
\noindent  {\it Dipartimento di Matematica, Universit\`a di Genova, Italy }\\[1mm]
email: \ {\tt  murro@dima.unige.it}
\\[10mm]
\end{center}

\begin{abstract}
In this paper, the Cauchy problem for a Friedrichs system on a globally hyperbolic manifold with a timelike boundary is investigated. By imposing admissible boundary conditions, the existence and the uniqueness of strong solutions are shown. Furthermore, if the Friedrichs system is hyperbolic, the Cauchy problem is proved to be well-posed in the sense of Hadamard. Finally, examples of Friedrichs systems with admissible boundary conditions are provided.
\end{abstract}

\\[0.5mm]

\tableofcontents

\renewcommand{\thefootnote}{\arabic{footnote}}
\setcounter{footnote}{0}

\section{Introduction}

The Cauchy problem for hyperbolic partial differential equations on curved 
spacetimes has been and continues to be at the forefront of scientific 
research. 
While for a generic spacetime the well-posedness of the Cauchy problem cannot be 
expected, in the class of \textit{globally hyperbolic manifolds} (with empty 
boundary) it has been proved that any hyperbolic PDE admits a unique smooth 
solution which depends continuously on the smooth Cauchy data, see the founding 
article \cite{Leray53} as well as \cite{Ba-lect,wave}.
Even though globally hyperbolic spacetimes have plenty of applications to physics, there
exist also important and interesting situations which require the spacetimes to have a non-trivial boundary. For example, experimental setups for
studying the Casimir effect confine quantum field theories between several 
metal plates, which may be modeled theoretically by introducing timelike 
boundaries to the system.
From a PDE viewpoint, this suggests that the
Cauchy problem could be well-posed once suitable boundary conditions are 
introduced.
In the last two decades, the well-posedness of the mixed 
initial-boundary problem for hyperbolic operators has been investigated in 
different geometric settings: see e.g. 
\cite{BenzoniGavageSerre07,SarbachTiglio12} for general surveys, 
\cite{Bachelot,Dappia1,Dappia4,Niky,GannotWrochna,
Friedrich95,Vasy,Warnick13} for asymptotically 
anti-de Sitter spacetimes, \cite{Dappia2,Dappia3,Felix} for the Klein-Gordon, 
the 
Maxwell and the Dirac operator on stationary spacetimes and \cite{DiracMIT,DiracAPS} for 
the case of Dirac operator on globally hyperbolic spacetimes. 

The aim of this paper is to prove well-posedness of the Cauchy problem, not only 
for hyperbolic PDE on globally hyperbolic manifolds with timelike boundary (\cf 
Definition~\ref{def:glob hyper manifold timelike boundary}), but for a larger 
class, known as \textit{Friedrichs systems} (\cf Definition~\ref{def:Fsyst}).
Friedrichs systems were developed by K.O.~Friedrichs 
in~\cite{Friedrichs0,Friedrichs} and include a large variety of PDE.
The 
classical Dirac operator is an example and many second-order PDEs (like wave 
equations and the heat equation) can be reduced to a Friedrichs system.
Our first main result is the existence of strong solutions for the forward 
Cauchy problem for Friedrichs system coupled with future admissible boundary 
conditions (\cf Definition~\ref{def:admissible bc}).
 \begin{theorem}[Strong solutions]\label{thm:main1}
 Let $\M$ be a globally hyperbolic 
manifold with timelike boundary and let 
 $t\colon\M\to\RR$ be a Cauchy temporal function.
 For any $0<T\in\RR$ denote 
with $\M_\T:=t^{-1}((0,T))$ a time strip.
Let  $\oS$ be a Friedrichs 
system with constant characteristic and 
denote with $G_\oB$ a {future} admissible 
boundary condition.
Then, there exists a strong solution of the Cauchy 
problem
\begin{equation}\label{CauchyM_T0}
\begin{cases}{}
{\oS }\Psi=\f \in \Gamma_c(\E_{|_{\M_\T}}) \\
\Psi_{|_{\Sigma_0}} = \h \in \Gamma_c(\E_{|_{\Sigma_0}})  \\
\Psi_{|_\bM}\in\oB:=\ker(G_\oB)\,.
\end{cases} 
\end{equation} 
Furthermore, if the 
bilinear form $\fiber{\sigma_\oS (dt) \cdot}{\cdot}_p$ is positive definite on
$\E_p$ for every $p\in\Sigma_\T$, then the solution is unique.
\black 
 \end{theorem}
 
 While full regularity of the strong solution cannot be expected for a generic 
Friedrichs system even for smooth Cauchy data (see Section~\ref{sec:regularity} 
for more details), our second main result shows that the backward and the forward Cauchy problem for 
symmetric hyperbolic systems coupled with admissible boundary condition is well-posed.  
 
\begin{theorem}[Smooth solutions]\label{thm:main2}
Let $\M$ be a globally hyperbolic 
manifold with timelike boundary and let $\oS$ be a
 symmetric  hyperbolic system of constant characteristic.
Assume $\oB=(\oB_+,\oB_-)$ to be an admissible boundary condition 
for $\oS$.
Let $\Sigma_0$ be any smooth spacelike Cauchy hypersurface of $\M$.
Then, for all $\f\in\Gamma_c(\E)$ and 
$\h\in\Gamma_c(\E_{|_{\Sigma_0}})$ 
satisfying the compatibility conditions~\eqref{eq:comp cond data} and 
\eqref{eq:comp cond 
databis} up to 
any 
order,  there exists a unique $\Psi\in\Gamma(\E)$ satisfying the Cauchy problem
\begin{equation}\label{Cauchy}
\begin{cases}{}
{\oS }\Psi=\f  \\
\Psi_{|_{\Sigma_{t_0}}} = \h   \\
\Psi_{|_{\bM\cap J^+(\Sigma_0)}}\in\oB_+\\
\Psi_{|_{\bM\cap J^-(\Sigma_0)}}\in\oB_-
\end{cases} 
\end{equation}
 and the map $(\f,\h)\mapsto\Psi$ 
sending a pair $(\f,\h)\in\Gamma_c(\E)\times\Gamma_c(\E_{|_{\bM}})$ to the 
solution $\Psi\in\Gamma(\E)$ of \eqref{Cauchy}, is continuous. 
 \end{theorem}

Roughly speaking, condition~\eqref{eq:comp cond data} up to some finite order 
$k$ ensures that, when the support of initial data meets the boundary of 
$\Sigma_0$, the solution of the Cauchy problem is $C^k$.
\medskip

Showing the well-posedness of the Cauchy problem is not the end of the story: 
Indeed an explicit construction of the evolution operator (as 
in~\cite{cap,Matte1,Matte2,Matte3,CapSimo} for the case of empty boundary) and a 
propagation of the singularity theorem should to be investigated. Clearly,  the 
well-posedness of the Cauchy problem will guarantee the existence 
of Green operators (\cf Proposition~\ref{prop:Green}) which play a pivotal role 
in the algebraic approach to 
linear quantum field theory, see e.g.~\cite{gerard, aqft2} for textbooks, 
\cite{CQF1,CQF2,FK,BD, ThomasAlex} for recent reviews, 
\cite{aqftCAT,aqftCAT2,aqftCAT3,aqftCAT4} for homotopical approaches 
and~\cite{Nic,Buc,cap,simo4,DHP,DMP3, CSth,cospi,simo1,simo2,simo3,SDH12} for 
some applications.
Indeed, they fully 
characterize the space of solutions of a symmetric hyperbolic system 
\cite{Ba,DHP}, they implement the canonical commutation/anticommutation 
relations typical of any linear quantum field theory \cite{CQF1,BD}, and their 
difference, dubbed the casual propagator or Pauli-Jordan commutator, can be 
used to construct quantum states, see e.g. 
\cite{simogravity,simo4,simo3,FK,simo1,simo2}. 

Our strategy to prove the well-posedness of the Cauchy problem is as follows:
first we derive a suitable energy inequality for a Friedrichs system in 
Section~\ref{sec:Energy Ineq} which will be employed in Section~\ref{sec:weak} 
to show existence and uniqueness of weak solutions.
In Section~\ref{sec:strong 
sol} we shall prove that any weak solution is actually a strong solution.
This 
will be achieved by localizing the problem and then using the theory of 
mollifiers, see e.g.~\cite{Rauch}.
In Section~\ref{sec:regularity}, we discuss 
the regularity of solutions of symmetric hyperbolic systems and in 
Section~\ref{sec:global well-posedness} we prove the well-posedness of the 
Cauchy problem.
Finally in Section~\ref{sec:Ex SHS} and Section~\ref{sec:Ex SPS} 
we provide some examples of Friedrichs systems with admissible boundary 
conditions.

 \subsection*{Acknowledgements}
 
We would like to thank Claudio Dappiaggi, Nicol\`o Drago, Christian G\'e\-rard, Nadine Gro\ss e, Valter Moretti and Alexander Strohmaier for helpful discussions related to the topic 
of this paper.
N.G. thanks Nadine Gro\ss{}e and the University of 
Lorraine for their 
support.  S.M. acknowledges the support of the INFN-TIFPA project ``Bell'' as 
well as the support of the University of Trento.

\subsection*{Notation and convention}
\begin{itemize}
\item[-] The symbol $\KK$ denotes one of the elements of the set $\{\RR,\CC\}$.
\item[-] $\M:=(\M,g)$ is a globally hyperbolic $n+1$-dimensional manifold with 
timelike boundary 
$\bM$ and we adopt the convention that $g$ has the signature $(-,+,\dots,+)$.
\item[-] $t:\M\to\RR$ is a Cauchy temporal function and 
$\M_\T:=t^{-1}(0,T)$ 
is a time strip.
\item[-] $\n$ is the outward unit normal vector to $\bM$.
\item[-] $\flat:\T\M\to \T^*\M$ and $\sharp:\T^*\M\to \T\M$ are the musical 
isomorphisms.
\item[-] $\E$ is a $\KK$-vector bundle over $\M$ with $N$-dimensional fibers, 
denoted by $\E_p$ for $p\in\M$, and endowed with a Hermitian fiber metric 
$\fiber{\cdot}{\cdot}_p$ .
\item[-]  $\Gamma_{c}(\E), \Gamma_{sc}(\E)$ resp. $\Gamma(\E)$ denote the 
spaces of compactly supported, spacelike compactly supported resp. smooth 
sections of $\E$.
\item[-] $\oS$ is a symmetric system of constant characteristic and 
$\oS^\dagger$ denotes the formal adjoint operator with respect to the fiber 
metric $\fiber{}{}_p$.
\item[-]  $G_\oB$  and $G_{\oB^\dagger}$ are (future) admissible boundary 
conditions for $\oS$ 
and $\oS^\dagger$ respectively and $\oB:=\ker\G_\oB$ and $\oB^\dagger:=\ker 
\G_{\oB^\dagger}$.
\end{itemize}

\section{Geometric preliminaries}

Let $\M$ be a connected, oriented, time-oriented smooth manifold with boundary.
We assume $\M$ to be endowed with a smooth Lorentzian metric $g$. 
Here and 
in the following we shall assume that the boundary is timelike, \ie
the pullback of $g$ with respect to the natural inclusion $\iota: \bM \to \M$ 
defines a 
Lorentzian metric $\iota^*g$ on the boundary.
In the class of Lorentzian 
manifolds with timelike boundary, those called globally hyperbolic provide a 
suitable background where to analyze the Cauchy problem  for hyperbolic 
operators. 
\begin{definition}{\protect{\cite[Definition 
2.14]{Ake-Flores-Sanchez-18}}}\label{def:glob hyper manifold timelike boundary}
A \emph{globally hyperbolic manifold with timelike boundary} is a $(n + 
1)$-dimensional, oriented,
time-oriented, smooth Lorentzian manifold $\M$ with timelike boundary $\bM$ 
such 
that
\begin{itemize}
\item[(i)] $\M$ is causal, \ie  there are no closed causal curves;
\item[(ii)] for every point $p,q\in\M$, $J^+(p)\cap J^-(q)$ is compact, where 
$J^+(p)$ (\textit{resp.} $J^-(p)$) denotes the causal future (\textit{resp}. 
past) of $p$ (\textit{resp.} q).
\end{itemize}
\end{definition}
\begin{remark}
 In case of an empty boundary, this definition agrees with the standard one, 
see e.g.~\cite[Section 3.2]{Bee} or \cite[Section 1.3]{wave}. 
\end{remark}
Recently, Ak\'e, Flores and S\'anchez gave a characterization of globally 
hyperbolic 
manifolds with timelike boundary:

\begin{theorem}[\cite{Ake-Flores-Sanchez-18}, Theorem 1.1]\label{thm: Sanchez}
Any globally hyperbolic manifold with timelike boundary admits a Cauchy 
temporal function $t\colon \M\to\RR$ with gradient tangent to $\bM$.
This implies that $\M$ splits into $\RR \times \Sigma$ with metric 
$$ g= - \beta^2 d t^2 + h_t $$
where $\beta : \RR \times \Sigma \to \RR$ is a smooth positive function, $h_t$ 
is a Riemannian metric on each slice
$\Sigma_t:=\{t\} \times \Sigma$ varying smoothly with $t$, and these slices are 
spacelike Cauchy hypersurfaces with boundary 
$\bSigma_t:=\{t\}\times\partial\Sigma$, namely achronal sets intersected 
exactly 
once by every inextensible timelike curve.
\end{theorem}

\subsection{Friedrichs systems of constant characteristic}

Let $\E\to\M$ be a Hermitian vector bundle over a globally hyperbolic manifold 
with timelike boundary $\M$, namely a $\KK$-vector bundle with finite rank $N$ 
endowed with a nondegenerate Hermitian fiber metric
$\fiber{\cdot}{\cdot}_p\colon\E_p\times\E_p\to\KK$.

\begin{definition}\label{def:symm syst}
 A linear differential operator $\oS \colon \Gamma(\E) \to \Gamma(\E)$ of first 
order is called a \emph{symmetric system} over $\M$ if 
\begin{enumerate}
\item[(S)] The principal symbol $\sigma_\oS (\xi) \colon \E_p \to \E_p$ is 
hermitian with respect to $\fiber{\cdot}{\cdot}_p$ for every $\xi\in \T^*_p\M$ 
and 
for every $p \in \M$.
\end{enumerate}
Additionally, we say that $\oS$ is \emph{hyperbolic} respectively 
\emph{positive} if it holds:
\begin{enumerate}
\item[(H)]\label{conditionH} For every future-directed timelike covector $\tau 
\in \T_p^*\M$, the 
bilinear form $\fiber{\sigma_\oS (\tau) \cdot}{\cdot}_p$ is positive definite 
on 
$\E_p$ for every $p\in\M$;
\item[(P)]\label{conditionP} For any Cauchy hypersurface $\Sigma_t\subset \M$, 
the 
quadratic form $\phi\mapsto\fiber{\Re e(\oS^\dagger 
+\oS)\phi}{\phi}$ on $\Sigma_t$ is 
uniformely bounded from below by a positive scalar multiple $c_t$, depending continuously on $t$, of the
quadratic 
form 
$\phi\mapsto\fiber{\phi}{\phi}$, which is assumed to be positive definite.
\end{enumerate}
\end{definition}
\begin{definition}\label{def:Fsyst}
We call \emph{Friedrichs system}, any symmetric system $\oS$ which is 
hyperbolic 
or positive.
Furthermore, we say that $\oS$ is \emph{of constant characteristic} if 
$\dim\ker \sigma_\oS(\n^\flat)$ is constant.
In particular, if $
\sigma_\oS(\n^\flat)$ has maximal rank we say that $\oS$ is \emph{nowhere 
characteristic}.
\end{definition}
\begin{remark}
Notice that Definition~\ref{def:symm syst} depends on the fiber metric 
$\fiber{\cdot}{\cdot}_p$. 
\end{remark}
\begin{example}
Consider the $n+1$-dimensional Minkowski spacetime 
$$\M=\RR \times \RR^n \qquad \eta=-dt^2+\sum_{j=1}^n dx_j^2$$ and let 
$\E:=\M\times \CC^N$ be a trivial vector bundle with the canonical fiber metric 
$\fiber{\cdot}{\cdot}_{\CC^N}$.
Any linear differential operator $\oS\colon \Gamma(\E) \to \Gamma(\E)$ of first 
order reads in a point $p\in \M$ as
\begin{equation*} 
\oS:=  A_0(p) \partial_t + \sum_{j=1}^n A_j(p) \partial_{x_j} + C(p)
\end{equation*}
where the coefficients $A_0, A_j,C$ are $N\times N$ matrices, with $N$ being 
the 
rank of $E$, depending smoothly on $p\in\M$.
In these coordinates, Condition~(S) in Definition~\ref{def:symm syst} reduces 
to 
$$A_0=A_0^\dagger \qquad \text{and} \qquad A_j=A_j^\dagger$$
 for $j=1,\dots, n$, where $\dagger$ is the complex conjugate of the transposed 
matrix.
Condition (H) and (P) can be stated respectively as follows:  
 $$(A_0 + \sum_{j=1}^{n} \alpha_j A_j) >0 \quad\text{ is positive definite for 
$\sum_{j=1}^{n} \alpha_j^2 <1$ ,}$$ 
 $$\Re e( C+C^\dagger - \frac{\partial_t (\sqrt{g}A_0)}{\sqrt{g}}- \sum_{j=1}^n 
\frac{\partial_{x_j} (\sqrt{g}A_j)}{\sqrt{g}})\quad \text{ is positive 
definite,} $$
  where $g$ is the absolute value of the determinant of the Lorentzian metric.
\end{example}

As we shall see in Section~\ref{sec:spin geom}, a prototype example of a first 
order system is the so-called classical Dirac operator. In this setting, the 
naturally defined fiber metric on the spinor bundle is indefinite rather than 
Hermitian.
It turns out that assuming the fiber metric to 
be positive-definite is not a loss of generality for a \emph{symmetric 
hyperbolic system}.
\begin{lemma}\label{lem:scalprod}
Let $\E$ be a $\KK$-vector bundle endowed with an indefinite nondegenerate 
sesquilinear fiber metric $\fiber{\cdot}{\cdot}_p$ and let $\oS$ be a 
symmetric  hyperbolic system
 with respect to $\fiber{\cdot}{\cdot}_p$. 
The operator $\fS_\beta:=\sigma_\oS(dt)^{-1} \oS$ is a symmetric hyperbolic
system
with respect to the positive-definite Hermitian fiber metric
\begin{equation}\label{eq:scalarprod}
\bracket{\cdot}{\cdot}_\beta:=\beta\fiber{\sigma_{\oS}(dt)\cdot}{\cdot}_p\,,
\end{equation}
 where $\beta:\M\to\RR^+$ is chosen on account of {\rm Theorem~\ref{thm: 
Sanchez}}.
Moreover, for any boundary space $\oB$, the Cauchy problem for the
operator $\fS_\beta$ 
  is equivalent to the Cauchy problem for $\oS$.
 \end{lemma}
\begin{proof}
On account of Properties (S),the fiber 
metric~\eqref{eq:scalarprod} is a Hermitian fiber metric.
In particular, for 
any $\xi\in \T^*_p\M$ it holds
\begin{eqnarray*}
\bracket{\sigma_\fS(\xi)\cdot}{\cdot}_\beta&=& 
\bracket{\sigma_\oS(dt)^{-1}\sigma_\oS(\xi)\cdot}{\cdot}_\beta\\
&=&\beta\fiber{\sigma_\oS(\xi)\cdot}{\cdot}_p\\
&=& 
\beta\fiber{\cdot}{\sigma_\oS(\xi)\cdot}_p \\
&=&\beta\fiber{\cdot}{\sigma_\oS(dt)\sigma_\oS(dt)^{-1}\sigma_\oS(\xi)\cdot}_p\\
&=& 
\bracket{\cdot}{\sigma_\oS(dt)^{-1}\sigma_\oS(\xi)\cdot}_\beta\\
&=& 
\bracket{\cdot}{\sigma_\fS(\xi)\cdot}_\beta\,,
\end{eqnarray*}
where we used Property (S) in the second and fourth equalities.
Moreover, any 
solution of the Cauchy problem for $\oS$  is a solution of the Cauchy problem 
for $\fS_\beta$ where the right-hand side is given by 
$(\frac{1}{\beta}\sigma_\oS(dt)^{-1}\f, \h)$.
\end{proof}

From now on, every time symmetric hyperbolic systems will come into play, the positive definite inner product which will be involved will be $\bracket{\cdot}{\cdot}_\beta$ from Lemma \ref{lem:scalprod}.
It will be denoted by $\fiber{\cdot}{\cdot}$ in order to keep notations simple.\\

The reader may wonder whether a symmetric system can be assumed to enjoy
property (P).
With the 
next lemma, we shall see that, at least on relatively compact subdomains, any 
symmetric hyperbolic
system can be transformed into a symmetric positive system such that the 
corresponding Cauchy problems remain equivalent.

\begin{lemma}\label{lem:sps in M_T}
Let $\M$ be a globally hyperbolic manifold with timelike boundary.
Let $t$ be a Cauchy temporal function and denote with 
$\M_\T$ a time strip, \ie $\M_\T:=t^{-1}(t_0,t_1)$. Finally let $\oS$ be a 
symmetric hyperbolic system.
 Then, for 
all $t_0,t_1\in\RR$ and for any $\lambda\in\RR$, the Cauchy problem for the 
symmetric system $\oK\colon \Gamma(\E_{|_{\M_\T}}) \to \Gamma(\E_{|_{\M_\T}})$ 
defined by
\begin{align*}\label{eq:K}
\oK:=  \oS + \lambda\sigma_{\oS}(dt)
\end{align*}
is equivalent to the Cauchy problem for $\oS$, namely
\begin{equation*}\label{CauchyK}
\begin{cases}{}
\oK\tilde\Psi=\tilde \f  \\
\tilde\Psi|_{\Sigma_0} = \tilde\h \\
\tilde\Psi \in \oB
\end{cases} 
\qquad\Longleftrightarrow\qquad
\begin{cases}{}
\oS \Psi=\f   \\
\Psi|_{\Sigma_0} ={\h} \\
 \Psi \in \oB,
\end{cases}
\end{equation*}
where $\tilde\f=e^{-\lambda t}\f$, $\tilde\h=\h$ and $\tilde\Psi=e^{-\lambda 
t}\Psi$. Moreover, for any relative compact set $U\subset \M$, there exists a 
constant $\lambda\equiv \lambda(U)$ such that $\oK$ is a positive symmetric 
system.
\end{lemma}
\begin{proof}
 For every 
$\Psi\in\Gamma(\E)$ and for every $t\in\RR$, we have
\begin{eqnarray*}
\oK(e^{-\lambda t}\Psi)&=&\left(\oS+\lambda\sigma_{\oS}(dt)\right)(e^{-\lambda 
t}\Psi)\\
&=&\sigma_{\oS}(d e^{-\lambda t})\Psi+e^{-\lambda 
t}\left(\oS+\lambda\sigma_{\oS}(dt)\right)\Psi\\
&=&-\lambda e^{-\lambda t}\sigma_{\oS}(dt)\Psi+e^{-\lambda 
t}\left(\oS+\lambda\sigma_{\oS}(dt)\right)\Psi\\
&=&e^{-\lambda t}\oS\Psi,
\end{eqnarray*}
which shows the correspondence between the Cauchy problems for $\oS$ and $\oK$.
By assumption, $\oS+\oS^\dagger$ is a zero-order operator and 
$\Psi\mapsto\fiber{\sigma_{\oS}(dt)\Psi}{\Psi}_p$ is positive definite on $\E$, 
therefore on every compact subset of $\M$ there exists a sufficiently large 
real $\lambda$ (depending on the compact set) such that the operator 
$\oK+\oK^\dagger=\oS+\oS^\dagger+2\lambda\sigma_S(dt)$ is positive definite. 
\end{proof}

\begin{remark}\label{r:assumptionoflemmalem:sps}
Actually the assumptions of Lemma \ref{lem:sps in M_T} may be weakeaned as 
follows: it is namely sufficient to assume $\oS$ to be symmetric, i.e. with 
$\sigma_{\oS}(\xi)^*=\sigma_{\oS}(\xi)$ for all $\xi\in\T^*\M$, and the family 
of 
pointwise quadratic forms $\Psi\mapsto\fiber{\sigma_{\oS}(dt)\Psi}{\Psi}$ to 
be 
uniformely bounded from below by a positive constant to get the result.
However those assumptions are equivalent to $\oS$ being symmetric hyperbolic 
for a perturbed Lorentzian metric $\bar{g}$ on $\M$.
For by continuity the quadratic form 
$\Psi\mapsto\fiber{\sigma_{\oS}(dt)\Psi}{\Psi}$ remains positive definite for 
all $\xi$ in an open neighborhood of $dt$ in $\T^*\M$.
We may assume without loss of generality that neighborhood to be an open cone 
in $\T^*\M$ which is contained in the set of future timelike covectors for $g$ 
and which depends smoothly on the base-point.
Now modifying the original metric $g$ only in $\partial_t$-direction, it is 
possible to obtain a new Lorentzian metric $\bar{g}$ such that its set of 
future timelike covectors coincides with -- or at least is in contained in -- 
the above cone neighborhood.
We may choose the same time orientation for that new metric $\bar{g}$.
Since its future cone is contained in the one of $g$, the new Lorentzian metric 
is globally hyperbolic (any timelike curve for $\bar{g}$ is a 
timelike curve for $g$) and $\oS$ becomes a symmetric hyperbolic operator on 
$(\M,\bar{g})$ by definition.
For the general discussion of perturbations of globally hyperbolic metrics, we 
refer to e.g. \cite[Sec. 4.2]{Ake-Flores-Sanchez-18}.
\end{remark}

We conclude this section, by deriving the Green identity for any first-order 
linear differential operator. To this end, consider the scalar product defined 
by
\begin{equation}\label{eq:L2}
  \scalar{\Phi}{\Psi}_{\M}
 :=\int_{\M} \fiber{\Phi}{\Psi} \, \vol_{\M}\,,
\end{equation}
 for all $\Psi,\Phi\in\Gamma(\E)$ such that 
$\supp\Psi\cap\supp\Phi$ 
is compact, where $\vol_{\M}$ is the metric-induced volume element.  

\begin{lemma}\label{lem:green id}
Let $\M$ be a 
manifold with Lipschitz boundary $\bM$ and $\oS$ be any first-order 
linear differential operator acting on sections of some Hermitian vector bundle 
$\E$ over $\M$.
Denote by $\oS^\dagger$ the formal adjoint of $\oS$.
Then for every $\Phi\in 
\Gamma_c(\E_{|_{\M}})$,
\begin{equation}\label{eq:Green Id}
\Re e\left(\scalar{\oS \Phi}{\Phi}_{\M} - 
\scalar{\Phi}{\oS^\dagger\Phi}_{\M}\right) = \Re e
\scalar{\Phi}{\sigma_{\oS}(\n^\flat) \Phi}_{\partial {\M} }\,,
\end{equation}
where $\n$ is the outward unit normal vector to $\partial{\M}$ and 
${}^\flat\colon \T\M\to \T^*\M$ denotes the musical isomorphism.
If furthermore $\oS$ is symmetric i.e., its principal symbol is Hermitian, then 
\eqref{eq:Green Id} holds without taking the real parts on both sides.
\end{lemma}
\begin{proof}
Let $\nabla$ be any metric covariant derivative on $\E$.
Let $b_0,\dots, b_n$ be a local tangent frame which is synchronous at the point 
under con\-si\-de\-ra\-tion, i.e. $\nabla b_j=0$, and denote with 
$b_0^*,\dots,b_n^*$ 
the dual basis. 
In such basis, the operator $\oS$ and its formal adjoint $\oS^\dagger$ read as
$$\oS= \sum_{j=0}^n \sigma_{\oS}(b^*_j)\nabla_{b_j} + C \,, \qquad \oS^\dagger= 
\sum_{j=0}^n - \sigma_{\oS}(b^*_j)^\dagger\nabla_{b_j} -\nabla_{b_j}\left(
\sigma_{\oS}(b^*_j)^\dagger\right) + C^\dagger \,, $$
where $C$ is some zero-order operator.
Consider now the real $n$-form on $\M$ given by
\begin{equation}\label{eq:diff form}
\omega:=\sum_{j=0}^n  \Re e\fiber{\sigma_{\oS}(b^*_j) \Phi}{\Phi}_p 
b_j\lrcorner \text{vol}_{\M}
\end{equation} 
where $\lrcorner$ denotes denotes the
insertion of a tangent vector into the first slot of a form.
By straightforward computation we get 
\begin{eqnarray*}
d\omega&=& \Re e \sum_{j=0}^n \Big(\fiber{ \nabla_{b_j} (\sigma_{\oS}(b^*_j))
\Phi}{\Phi}_p +  \fiber{  \sigma_{\oS}(b^*_j) \nabla_{b_j}\Phi}{\Phi}_p \\
& &\phantom{\Re e \sum_{j=0}^n \Big(}- 
\fiber{ 
\Phi}{-  \sigma_{\oS}(b^*_j)^\dagger\nabla_{b_j}\Phi}_p \Big) 
\text{vol}_{\M}\\
&=& \Re e\left( \fiber{\oS \Phi}{\Phi}_p - 
\fiber{\Phi}{\oS^\dagger\Phi}_p\right) 
\text{vol}_{\M}.
\end{eqnarray*}
Using Stokes' theorem for 
manifolds with Lipschitz boundary we obtain \eqref{eq:Green Id}.
Note that $\n^*=\n^b$ along $\bM$ because of $\n$ being spacelike.
\end{proof}
\begin{remark}
In case $\oS$ is symmetric, the differential form $\omega$ defined above is 
real, therefore we obtain \eqref{eq:Green 
Id} without the real parts on both sides.
\end{remark}

\subsection{Admissible boundary conditions}\label{sec:admissible bc}

In this paper we are interested in sections subject to certain linear 
homogeneous boundary conditions, depending of course if we want to solve the 
forward or the backward Cauchy 
problem.
We begin by fixing a Cauchy surface $\Sigma_0:=t^{-1}(\{0\})$ 
where we shall assign the initial data.
\color{black}
 To define these boundary conditions we 
associate with each boundary point $q\in\bM$ a pair of linear 
subspaces $(\oB_{\pm})_q\subset 
\E_q$  whose dimensions are the same at all points of $\bM$ and which vary 
smoothly with $q$\color{black}.
In particular, we shall focus on a class 
introduced by 
Friedrichs and Lax-Phillips in~\cite{Friedrichs,Lax-Phillips}, dubbed 
admissible 
boundary conditions. 

\begin{definition}\label{def:admissible bc}
A smooth linear bundle map $\G_{\oB_+}\colon\E_{|_\bM 
}\rightarrow\E_{|_\bM}$ is 
said to be a \emph{future admissible boundary condition} for a first-order 
Friedrichs system $\oS$ if 
\begin{itemize}
\item[(i-f)] the pointwise kernel $\oB_+$ of $\G_{\oB_+}$ is a 
smooth subbundle of $\E_{|_\bM}$;
\item[(ii-f)] the quadratic form
$\Psi\mapsto\fiber{\sigma_{\oS}(\n^\flat)\Psi}{\Psi}_p$ is positive 
semi-definite 
on $\oB_+$ ;
\item[(iii-f)] the rank of $\oB_+$ is equal to the 
number of 
pointwise non-negative eigenvalues of 
$\sigma_{\oS}(\n^\flat)$ counting multiplicity.
\end{itemize}
 Similarly we say that $\G_{\oB_-}\colon\E_{|_\bM}\rightarrow\E_{|_\bM}$ is 
\emph{past admissible}
if
\begin{itemize}
\item[(i-p)] the pointwise kernel $\oB_-$ of $\G_{\oB_-}$ is a 
smooth subbundle of $\E_{|_\bM}$;
\item[(ii-p)] the quadratic form
$\Psi\mapsto\fiber{\sigma_{\oS}(\n^\flat)\Psi}{\Psi}_p$ is negative 
semi-definite 
on $\oB_-$;
\item[(iii-p)] the rank of $\oB_-$ is equal to the 
number of 
pointwise non-positive eigenvalues of 
$\sigma_{\oS}(\n^\flat)$ counting multiplicity.
\end{itemize}
The pair $\oB=(\oB_+, \oB_-)$ is called the \emph{admissible boundary space} or 
\emph{admissible boundary condition}
for $\oS$.
\end{definition}

\begin{remark}
The role of $\oB_+$ and $\oB_-$ will become apparent when 
looking for energy estimates for symmetric hyperbolic $\oS$, see Theorem 
\ref{thm:Energy Ineq}.
It turns out that $\oB_+$ (resp. $\oB_-$) is only needed in the future (resp. 
past) of the chosen Cauchy hypersurface $\Sigma_0$.
\end{remark}
\color{black}

Notice that if $\fiber{\cdot}{\cdot}$ is not positive definite, by 
Lemma~\ref{lem:scalprod} the new symmetric hyperbolic system 
$\fS_\beta$ together with the Hermitian positive-definite fiber metric 
$\bracket{\cdot}{\cdot}_\beta$ can be defined such that the Cauchy problems for 
both $\fS_\beta$ and $\oS$ become equivalent.
In particular, being an admissible 
boundary condition for $\fS_\beta$ is equivalent to be admissible for $\oS$. 
Indeed 
it holds
$$ 
\bracket{\sigma_{\fS_\beta}(\n^\flat)\Psi}{\Psi}_\beta\,=\,\fiber{\sigma_{\oS}
(\n^\flat)\Psi}{\Psi} \,.$$

Conditions (ii-f) and (ii-p) are equivalent to require that the boundary conditions are
\emph{maximal} with respect to properties  (iii-f) and (iii-p) respectively, 
see 
\cite[Theorem D.1]{Lax2006}, namely no smooth vector subbundles 
$(\oB')_\pm$ of $\E$ exist that properly contains $\oB_\pm$ and such that for 
all  
$\Phi'\in(\oB')_+$ and $\Phi''\in(\oB')_-$
$$\fiber{\sigma_{\oS}(\n^\flat)\Phi'}{\Phi'}\geq 0 \qquad 
\fiber{\sigma_{\oS}(\n^\flat)\Phi''}{\Phi''}\leq 0$$  
holds. 
\color{black} 
The fact that we do not assume $\sigma_{\oS}(\n^\flat)$ to be invertible 
(which is the case in \cite[App. D]{Lax2006}) does not play any role.
As a consequence, note that  
$\ker(\sigma_\oS(\n^\flat))\subset\oB_+\cap\oB_-$. \color{black}

\begin{definition}\label{def:adjointbc}
Let $\G_\oB$ be a future or past admissible boundary condition for a given 
first-order 
Friedrichs system $\oS$ on $\E$.
Assume $\oS$ to be of constant characteristic along $\bM$.
The \emph{adjoint boundary condition} $\G_\oB^\dagger$ is defined as the 
pointwise orthogonal projection onto $\sigma_{\oS}
(\n^\flat)(\oB)$.
In particular,
\[\oB^\dagger:=\ker(\G_\oB^\dagger)=\left(\sigma_{\oS}
(\n^\flat)(\oB)\right)^\perp.\]
If $\oB^\dagger=\oB$ then we say that $\oB$ is a {\em self-adjoint future/past 
admissible boundary space}.
\end{definition}

Similarly to \cite[Theorem D.2]{Lax2006}, it can be shown that,  if 
$\oB$ is a future admissible boundary condition for instance\color{black}, 
then the quadratic 
form 
$\Phi\mapsto\fiber{\sigma_{\oS}(\n^\flat)\Phi}{\Phi}_p$ is negative 
semi-definite on 
$\oB^\dagger$,
 whose rank coincides with the number of nonpositive  
eigenvalues of 
$\sigma_{\oS}(\n^\flat)$ counted with multiplicities.
Namely $\ker(\sigma_{\oS}(\n^\flat))$ must be contained in $\oB$ by its 
maximality property, so that, pointwise,
\begin{eqnarray*}
\dim(\oB^\dagger)&=&\dim(\E_{|_{\bM}}
)-\dim(\oB)+\dim(\ker(\sigma_ { \oS } (\n^\flat))\cap\oB)\\
&=&\dim(\E_{|_{\bM}}
)-\dim(\oB)+\dim(\ker(\sigma_ { \oS } (\n^\flat))),
\end{eqnarray*}
which is precisely the number of nonpositive  
eigenvalues of 
$\sigma_{\oS}(\n^\flat)$ counted with multiplicities.
This in turn implies that $\oB^\dagger$ is maximal such that 
$\Phi\mapsto\fiber{\sigma_{\oS}(\n^\flat)\Phi}{\Phi}_p$ is negative 
semi-definite on $\oB^\dagger$,
\color{black} because pointwise any subbundle containing 
$\oB^\dagger$ and enjoying the same property cannot intersect the subbundle 
spanned by the eigenvectors associated to the positive
 eigenvalues of 
$\sigma_{\oS}(\n^\flat)$ in a nontrivial way. 
Therefore it must have the same 
dimension as -- and therefore coincide with -- $\oB^\dagger$.
As a consequence, if $\oB$ is future admissible for a given Friedrichs systems 
$\oS$, 
then $\oB^\dagger$  is future admissible \color{black} for 
$\oS^\dagger$.
Moreover, 
by construction of 
$\oB^\dagger$,
 for all 
$(\Psi,\Phi)\in \oB\times_{\bM}\oB^\dagger$ it holds
\[\fiber{\sigma_{\oS}(\n^\flat)\Psi}{\Phi}_p=0\,.\]

\subsection{The forward and the backward Cauchy problem}

The backward Cauchy problem for a symmetric 
hyperbolic system $\oS$ is equivalent to the forward Cauchy problem for 
$-\oS$ on the time-reversed underground spacetime:

\begin{lemma}\label{l:timereversalSsymmhyp}
Let $t\colon\M\to\RR$ be a Cauchy temporal function with gradient tangent to 
the boundary on a given globally hyperbolic spacetime with timelike boundary 
$\M$.
Let $\oS$ be any symmetric hyperbolic system of constant characteristic along 
$\bM$ and with admissible boundary space $(\oB_+,\oB_-)$ along $\bM$.\\
Then $\bar{\oS}:=-\oS$ is symmetric hyperbolic on $\bar{\M}$, which is $\M$ 
with the same metric but with reversed time orientation.
Moreover, $(\oB_-,\oB_+)$ is an admissible boundary space for $\bar{\oS}$ along 
$\bM$.
\end{lemma}

\begin{proof}
A $1$-form $\xi\in \T^*\bar{\M}=\T^*\M$ is future-oriented causal on $\bar{\M}$ 
if and only if it is past-oriented causal on $\M$.
Therefore, if $\xi$ is future-oriented causal on $\bar{\M}$, then 
$\sigma_{\bar{\oS}}(\xi)=\sigma_{\oS}(-\xi)$ is a pointwise positive definite 
endormorphism of $\E$ because $-\xi$ is future-oriented on $\M$.
This shows the first statement.\\
The second statement follows from the fact that the outward unit normal $\n$ to 
$\bM$ does not change when the time orientation is reversed, so that, if 
$\psi\mapsto\fiber{\sigma_{\oS}(\n^\flat)\psi}{\psi}$ is positive 
(resp. negative) semi-definite on some subbundle of $\E_{|_{\bM}}$, then 
$\psi\mapsto\fiber{\sigma_{\bar{\oS}}(\n^\flat)\psi}{\psi}=-\fiber{\sigma_{\oS} 
(\n^\flat)\psi}{\psi}$ is negative (resp. positive) semi-definite on that same 
subbundle.
\end{proof}

As a consequence of Lemma \ref{l:timereversalSsymmhyp}, given any 
symmetric hyperbolic system $\oS$ on $\M$ with admissible boundary space 
$(\oB_+,\oB_-)$ and given any $\f\in\Gamma_c(\E)$ as well as 
$\h\in\Gamma_c(\E_{|_{\Sigma_0}})$, solving

\begin{equation}\label{eq:ibvpsymmhypgeneral}
\left\{\begin{array}{lll}\oS u&=\f&\textrm{ on }\M\\ 
 u_{|_{\Sigma_0}}&=\h&\textrm{ on }\Sigma_0\\ 
u_{|_{J_{\M}^+(\Sigma_0)\cap\bM}}&\in\oB_+&\textrm{ along 
}J_{\M}^+(\Sigma_0)\cap\bM\\ 
u_{|_{J_{\M}^-(\Sigma_0)\cap\bM}}&\in\oB_-&\textrm{ along 
}J_{\M}^-(\Sigma_0)\cap\bM\end{array}\right.
\end{equation}

on $\M$ is equivalent to solving 

\begin{equation}\label{eq:ibvpsymmhypgeneraltimereversed}
\left\{\begin{array}{lll}-\oS u&=-\f&\textrm{ on }\bar{\M}\\ 
 u_{|_{\Sigma_0}}&=\h&\textrm{ on }\Sigma_0\\ 
u_{|_{J_{\bar{\M}}^+(\Sigma_0)\cap\bM}}&\in\oB_-&\textrm{ along 
}J_{\bar{\M}}^+(\Sigma_0)\cap\partial\bar{\M}\\ 
u_{|_{J^-(\Sigma_0)\cap\bM}}&\in\oB_+&\textrm{ along 
}J_{\bar{\M}}^-(\Sigma_0)\cap\partial\bar{\M}\end{array}\right.
\end{equation}
on $\bar{\M}$: a section $u$ of $\E$ solves \eqref{eq:ibvpsymmhypgeneral} on 
$\M$ if and only if $u$ solves \eqref{eq:ibvpsymmhypgeneraltimereversed} on 
$\bar{\M}$.\\

\section{Energy Inequality}\label{sec:Energy Ineq}

In this section we derive a suitable energy inequality for 
Friedrichs systems in any time strip $\M_\T:=t^{-1}((0,T))$.
\color{black}
 By denoting with $\|\cdot\|_{L^2(\E_{|_{\M_\T}})} $ the  norm 
corresponding to the scalar product $\scalar{\cdot}{\cdot}_{\M_\T}$ defined by 
Equation~\eqref{eq:L2}, the main result of this section is the following.

\begin{theorem}[Energy Inequality]\label{thm:energyest}
Let $\M$ be a globally hyperbolic manifold
with timelike boundary, let $t\colon\M\to\RR$ be a Cauchy temporal 
function.
Let $\M_\T$ be the time strip $\M_\T:=t^{-1}((0,T))$.
Let 
$\oS$ be a Friedrichs system and denote by $\oS^\dagger$  the formal adjoint 
operator.
Assume $\M$ to be Cauchy-compact when $\oS$ is symmetric 
hyperbolic. 
Finally denote by $\G_\oB$ a future admissible boundary 
condition and by $\G_\oB^\dagger$ the adjoint 
boundary condition.
Then there exists a positive constant $\tilde C=\tilde{C}(\M_\T)$ such that, for all 
$\Phi\in \Gamma_c(\E_{|_{\M_\T}})$ satisfying 
$\Phi|_{\Sigma_{0}}=0$, $\Phi|_{\Sigma_{T}}=0$ and $\Phi_{|_{\bM_T}} \in
\oB^\dagger|_{\M_\T}$,
\begin{equation}\label{Energy Inequality2}
\|\Phi\|_{L^2(\E_{|_{\M_\T}})} \leq \tilde C \| \oS^\dagger 
\Phi\|_{L^2(\E_{|_{\M_\T}})}\,.
\end{equation}
Furthermore, if the bilinear form $\fiber{\sigma_\oS (dt) \cdot}{\cdot}_p$ is positive definite on
$\E_p$ for every $p\in\Sigma_\T$, then there exists a constant $\tilde D=\tilde{D}(\M_\T)>0$ such that, for all
$\Psi\in \Gamma(\E_{|_{\M_\T}})\cap L^2(\E|_{\M_\T})$ satisfying 
$\Psi|_{\Sigma_{0}}=0$ and $\Psi_{|_{\bM_T}} \in
\oB|_{\M_\T}$,
\begin{equation}\label{Energy Inequality2}
\|\Psi\|_{L^2(\E_{|_{\M_\T}})} \leq \tilde D \| \oS
\Psi\|_{L^2(\E_{|_{\M_\T}})}\,.
\end{equation}
\end{theorem}
Before proving our claim, we need some preliminary results on symmetric 
hyperbolic systems.
Let $t\colon \M\to \RR$ be a Cauchy temporal function and set 
$\Sigma_{t}^p := J^- (p) \cap \Sigma_{t}$  for $p\in \M$.
Recall that
$\bracket{\cdot}{\cdot}_\beta$ denotes the normalized Hermitian scalar
product~\eqref{eq:scalarprod} from Lemma \ref{lem:scalprod}. Let $| \cdot |_\beta$ be the corresponding
norm.
Finally, let $d\mu_{t}$ be the volume density of $\Sigma_{t}$.

\begin{theorem}[Energy estimates for symmetric hyperbolic 
systems]\label{thm:Energy Ineq}
Let $\M$ be a globally hyperbolic manifold with 
timelike boundary and let $\oS$ be a symmetric hyperbolic system of constant 
characteristic. 
Then, for each $p \in \M$ and all $t_0,t_1\in\RR$ with $t_0\leq 
t_1$ there 
exists 
a constant $C=C(p,t_0,t_1) > 0$ such that
\begin{equation}\label{eq:energy ineq}
\int_{\Sigma_{t_1}^p} |\Psi|^2_\beta d\mu_{t_1} \leq C e^{C(t_1 -t_0 )}  
\int_{t_0}^{t_1}\int_{\Sigma_{s}^p} |\oS \Psi|^2_\beta d\mu_{s} ds + e^{C(t_1 
-t_0 )}  \int_{\Sigma_{t_0}^p}  |\Psi|^2_\beta d\mu_{t_0} 
\end{equation}
holds for each $\Psi\in \Gamma(\E)$ satisfying 
$\Psi_{|_{\bM}}\in\ker\G_\oB$, where $G_\oB$ is a future admissible boundary 
condition 
for $\oS$.
In particular, $C=C(t_0,t_1)$ if $\M$ is Cauchy-compact. 
\end{theorem}
\begin{proof}
 We shall reduce to the proof of~\cite[Theorem 5.3]{Ba}. To this 
end, let us define the subset $K:=J^- (p)\cap t^{-1}([t_0 ,t_1]) \subset \M$  
and consider the $n$-differential form defined by Equation~\eqref{eq:diff form}.
Stokes' theorem for manifold with Lipschitz boundary yields
$$\int_K d\omega=\int_{\partial K} \omega = \int_{\Sigma^p_{t_1}}\omega - 
\int_{\Sigma^p_{t_0}} \omega + \int_{K\cap \bM} \omega +  \int_{ Y} \omega $$
where $Y:= \partial J^-(p) \, \cap \, t^{-1}([t_0,t_1])$. In order to reduce 
our 
proof to the one of~\cite[Theorem 5.3]{Ba} we only need to show that
$$\int_{K \cap \bM} \omega \geq 0 \,.$$
We choose a positively oriented orthonormal tangent basis $b_0, b_1 , \dots , 
b_n$ of $\T_q \M$ in such a way that $b_0=-\frac{1}{\beta}\partial_t$ and 
$b_1=\n$, so that the restriction of $\omega$ to $\bM$ is given by 
$$\iota^*\omega=\fiber{\sigma_\oS(\n^\flat)\Psi}{\Psi}_p\n\lrcorner\vol_{K\cap\M
}
=\fiber{\sigma_\oS(\n^\flat)\Psi}{\Psi} \vol_{K\cap\bM}\,.$$
Therefore
$$\int_{K\cap\bM}\omega= \int_{K\cap\bM} 
\fiber{\sigma_\oS(\n^\flat)\Psi}{\Psi}_p 
\vol_{K\cap\bM}\,.$$
Since $\Psi|_\bM\in\oB|_{\M_\T}$, property (i-f) of
Definition~\ref{def:admissible bc} implies that the r.h.s. of the last 
identity is nonnegative, which concludes the proof.
\end{proof}

Combining Theorem~\ref{thm:Energy Ineq} with 
Lemma~\ref{l:timereversalSsymmhyp}, we immediately obtain 
that if there exists a solution to the Cauchy problem~\eqref{CauchyM_T0} it 
must 
be unique and it propagates with at most the speed of light. We recall this 
results 
for the sake of completeness.

\begin{proposition}[Uniqueness and finite speed of 
propagation for symmetric hyperbolic system]\label{prop:finite}
Let $\M$ be a globally hyperbolic manifold with 
timelike boundary and let $\oS$ be a symmetric hyperbolic system of constant 
characteristic coupled with admissible boundary conditions.
If there exists $\Psi \in \Gamma(\E|_{\M_\T})$ satisfying the Cauchy 
problem~\eqref{CauchyM_T0} then it is unique and it propagates with at most the 
speed of light, \ie its support is contained inside the region 
$$ \mathcal{V}:= \Big( J\big(\supp\, \f \big)\cup J(\supp \h)\Big)\,,$$ 
where $J(\cdot)$ denote the causal future of a 
set.
\end{proposition}
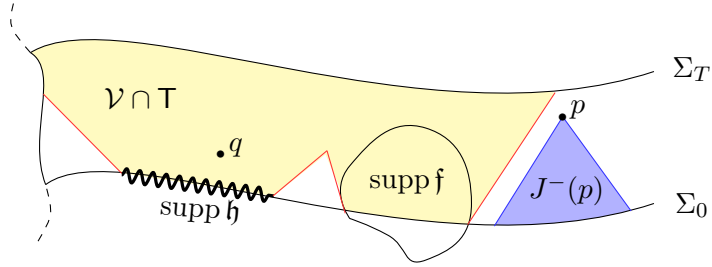
\begin{figure}[h!]
\centering
\begin{tikzpicture}
\begin{scope}       
    \clip {(9,1.75) .. controls (6,0.75) and (2,2.75).. (0,1.75) -- (0.8,2)  .. 
controls (1.2,1.75)  and (0.75,0.9) .. (1.0,0.25)  -- (1.0,0.25) .. controls 
(2,1) and (6,-1) .. (9,0) };    
    \fill[yellow!30] { plot[smooth, tension=1] coordinates { (5,0.8-1) 
(5.5,0.58-1) (6,0.25-1) (6.6,1-1) (6,2-1) (5, 1.5-1.1)   (5,0.8-1)  }};
;
                \node  at (5.5,1) {$\supp \f$};
\end{scope}

\begin{scope}    
    \clip {(9,1.75) .. controls (6,0.75) and (2,2.75).. (0.8,2) -- (0.8,2)  .. 
controls (1.2,1.75)  and (0.75,0.9) .. (1.0,0.25)  -- (1.0,0.25) .. controls 
(2,1) and (5,-1) .. (9,0) };     
    \fill[yellow!30] { (2,0.4) -- (0.97,1.45) -- (0.8,2) -- (0.8,2) .. controls 
(1,3).. (7.7,1.47)  -- (6.55,-0.28) -- (5,-0.1) -- (4.68,0.65) -- (4,0.1) -- 
(2,0.4)};
;
\end{scope}

\begin{scope}        
    \clip {(0,0) .. controls (2,1) and (6,-1) .. (9,0) -- (9,1.75) -- (9,1.75) 
.. controls (6,0.75) and (2,2.75).. (0,1.75) -- (0,0)};    
    \fill[blue!30] { (7.8,1.15) -- (6.9,-0.3)  .. controls ( 7.5, -.5) .. 
(8.7,-0.1) --(7.8,1.15)  };
;
                \node  at (5.75,0.3) {$\supp \f$};
\end{scope}

 \draw[]  (1.0,0.25) .. controls (2,1) and (6,-1) .. (9,0) ; 
 \draw[]  (0.8,2) .. controls (2,2.75) and (6,0.75) .. (9,1.75); 
  \draw[]  (1.0,0.25) .. controls (0.75,0.9) and (1.2,1.75) .. (0.8,2); 
    \draw[dashed]  (0.9,-0.5) .. controls (1.1,-0.15) .. (1.0,0.25); 
 
     \draw[dashed]  (0.8,2) .. controls   (0.5,2.35) .. (0.6,2.75); 

  \draw[decorate, decoration=snake, segment length=5,very thick] (2,0.38)  -- ( 
4,0.1)  node[midway,below,rotate=-5] {$\,\,\supp \h$}  ; 

            \draw[]  plot[smooth, tension=1] coordinates { (5,0.8-1) 
(5.5,0.58-1) (6,0.25-1) (6.6,1-1) (6,2-1) (5, 1.5-1)   (5,0.8-1) };

\draw[red!80] (2,0.4) -- (0.97,1.45);
\draw[red!80] (3.99,0.1) -- (4.7,0.7);
\draw[red!80] (4.94,-0.13) -- (4.7,0.7);
\draw[red!80] (6.55,-0.28) -- (7.7,1.47);
\draw[blue!80] (7.8,1.15) -- (6.9,-0.3);
\draw[blue!80] (7.8,1.15) -- (8.7,-0.1);

            \node  at (8,0.85) [label={$p$}] {};
            \node  at (7.8,1.15) {$\mathsmaller{\mathsmaller{\bullet}}$};

            \node  at (3.5,0.35) [label={$q$}] { };
            \node  at (3.3,0.65) {$\mathsmaller{\mathsmaller{\bullet}}$};
            \node  at (2.25,1) [label={${\mathcal{V}\cap \T}$}] { };
            \node  at (7.85,-0.3) [label={$J^-(p)$}] { };
            \node  at (9.5,-0.45) [label={$\Sigma_{0}$}] { };
            \node  at (9.5,1.35) [label={$\Sigma_{T}$}] { };

    \end{tikzpicture}
    \caption{Finite propagation of speed -- $\mathcal{V}\cap \T$.}
    \label{fig:finite speed}
\end{figure}
\begin{proof}
Assume $q\in J^+(\Sigma_{0})$  and consider any point $p$ outside the region 
$\mathcal{V}\cap \M_\T$, with $\M_\T:=t^{-1}(0,T)$ -- \cf 
Figure~\ref{fig:finite speed}.
This means that there is no future-directed causal
curve starting in $\supp\, \f \cup \supp\, \h $, entirely contained in 
$\mathcal{V}\cap\M_\T$, which terminates at $p$. 
As a consequence,
$\f|_{J^-(p)}\equiv 0$ and $\h|_{J^-(p)\cap\Sigma_{0}}\equiv 0$. Therefore, 
by 
Theorem~\ref{thm:Energy Ineq}, it follows that $\Psi$ vanishes in $J^-(p)$. 

The case $q\in J^-(\Sigma_{0})$ is obtained with the time reversal (see also 
Lem\-ma~\ref{l:timereversalSsymmhyp}).
Hence, $\Psi$ vanishes outside $\mathcal{V}$. 
\medskip

Assume that there exist $\Psi$ and $\Phi$ satisfying the same Cauchy 
pro\-blem \eqref{CauchyM_T0}. Then $\Psi-\Phi\in \Gamma(\E_{|_{\M_\T}})$ is a 
solution 
of \eqref{CauchyM_T0} with $\f=0$ and $\h=0$.
As we have already shown, the supports of $\Psi$ and $\Phi$ are contained in 
$\mathcal{V}\cap\M_\T$. Therefore, we can use Theorem~\ref{thm:Energy Ineq} to 
conclude that $\Psi-\Phi$ vanishes identically.
\end{proof}

We notice that, as in the boundaryless case, solving the Cauchy 
problem associated to a symmetric hyperbolic system for Cauchy-compact or 
arbitrary globally hyperbolic manifolds with timelike boundary are 
equivalent.
\begin{proposition}\label{Cauchycompactornotisthesameforsymmhyp}
Let $\M$ be a Cauchy-noncompact globally hyperbolic manifold with timelike 
boundary and let $(\M,g)=(\RR\times\Sigma_0,-\beta^2 dt^2\oplus h_t)$ be a 
splitting as in {\rm Theorem \ref{thm: Sanchez}}.
If, for a given symmetric hyperbolic system $\oS$ on $\M$, the Cauchy 
problem~\eqref{CauchyM_T0} can be solved on $(\RR\times U,-\beta^2 
dt^2\oplus h_t)$ for any relatively compact domain with smooth boundary 
$ U\subset\Sigma_0$ and any admissible boundary condition along 
$\partial U$, then so can it on $\M$ itself.
\end{proposition}

\begin{proof}
The proof virtually coincides with that of \cite[Theorem 3.7.7]{Ba-lect}.
Let $\f,\h$ be the Cauchy data in \eqref{CauchyM_T0} and set 
$K:=\supp(\f)\cup\supp(\h)\subset\M$.
Then $K$ is a compact subset of $\M$ and therefore is included in some 
$\M_\T=(0,T)\times\Sigma_0$ for some real $0<T$.
Let $\hat{\Sigma}$ be the projection onto $\Sigma_0$ of the compact subset 
$J^{\M}(K)\cap([0,T]\times\Sigma_0)$ w.r.t. the splitting 
$\M=\RR\times\Sigma_0$.
Then there exists a relatively compact open neighborhood $U$ of 
$\hat{\Sigma}$ in $\Sigma_0$
with smooth boundary $\partial U$.
Note that part of the boundary of $U$ is contained in $\bM$ and it is only on 
that part that the support of the solution to \eqref{CauchyM_T0} may meet 
$\partial U$.
Consider now $\M':=\RR\times\overline{U}$ with metric $g':=-\beta^2 dt^2\oplus 
h_t$, where $\overline{U}$ is the closure of $U$ in $\Sigma_0$.
Then $\M'$ is a new Lorentzian manifold and is globally hyperbolic because of 
$\overline{U}$ being compact: it can be directly shown that every inextendible 
timelike curve in $\M'$ meets $\overline{U}$ exactly once; the main point is 
that, on $\overline{U}$, all metrics $h_t$ for $t$ in a compact interval are 
uniformly equivalent to some fixed metric.
We refer to the proof of \cite[Lemma A.5.14]{wave} that can be adapted to our 
situation.
Now $\M'$ is a Cauchy-compact globally hyperbolic manifold with timelike 
boundary, therefore there exists a unique solution $\psi$ to the Cauchy problem 
\eqref{CauchyM_T0} on $\M'_\T:=(0,T)\times\overline{U}$.
Since by finite propagation speed the support of $\Psi$ is contained in 
$J(K)\cap\M'_\T$, it meets $\bM'$ only along $\bM$ and, since by construction 
$\Psi$ must vanish along the rest of $\partial U$, the section $\Psi$ can be 
considered as a section of $\E$ on $\M_\T$.
Therefore there exists a solution $\Psi$ to \eqref{CauchyM_T0} on $\M_\T$.
Uniqueness holds on $\M_\T$ as well since it holds on $\M_\T'$.
Since this holds for any $0<T\in\RR$, we obtain global existence and 
uniqueness of solutions to \eqref{CauchyM_T0} on $\M$.
\end{proof}

\begin{remark}\label{rmk:SHPS}
Assume $\oS$ to be a symmetric hyperbolic system on a globally hyperbolic 
manifold with noncompact Cauchy hypersurfaces. On account of 
Proposition~\ref{Cauchycompactornotisthesameforsymmhyp}, to solve the Cauchy 
problem for $\oS$ in a time strip $\M_\T$ it is enough solving it on 
$[0,T]\times \overline U$, where $\overline{U}$ is compact. Since 
$[0,T]\times  \overline U$ is compact, then Lemma~\ref{lem:sps in M_T} 
guarantees the existence of a suitable $\lambda$ such that the operator $\oK$ 
is a symmetric positive system and it has an equivalent Cauchy problem. Summing 
up, the Cauchy problem for a symmetric hyperbolic system $\oS$ on a globally 
hyperbolic manifold $\M$ with non-compact Cauchy surfaces can be solved if the 
Cauchy problem for $\oK$ can be solved on $[0,T]\times\overline{U}$.
Therefore, we may prove existence and uniqueness for the Cauchy 
problem for $\oS$ via the auxiliary operator $\oK$ for sufficiently large 
$\lambda$.
Note that this works only if $\M$ is Cauchy-compact.
\end{remark}

We now have all the ingredients to prove Theorem~\ref{thm:energyest}. 

\begin{proof}[Proof of Theorem~\ref{thm:energyest}]
First, the proof of Theorem~\ref{thm:energyest} for a symmetric 
hyperbolic system on a Cauchy-compact globally hyperbolic manifold is an 
immediate consequence of 
Theorem~\ref{thm:Energy Ineq} since it suffices to integrate \eqref{eq:energy 
ineq} on $[0,T]$ after choosing $-\oS^\dagger$ -- which is again symmetric 
hyperbolic -- instead of $\oS$. 

Let $\oS$ be now a symmetric positive system of constant characteristic. By Lemma~\ref{lem:green id}, Green identity reads 
$$
\scalar{\Phi}{ \oS \Phi }_{\M_\T} - \scalar{\oS^\dagger \Phi }{ \Phi }_{\M_\T}  
= 
 \scalar{\Phi}{\sigma_\oS(\n^\flat)\Phi }_{\partial\M_\T}   
$$
where  $\scalar{\cdot}{\cdot}_{\partial\M_\T}$ is the induced $L^2$-product on 
${\partial\M_\T}$. 
By adding $ 2 \scalar{\oS^\dagger \Phi }{  \Phi}_{\M_\T}$ and  using that $\oS$ 
is a 
symmetric positive system we thus obtain

\begin{equation}\label{passaggio1}
\begin{aligned}
 \scalar{\Phi}{  \sigma_\oS(\n)\Phi }_{\partial {\M_\T} } +  
{2} \scalar{\oS^\dagger\Phi}{ \Phi}_{\M_\T}  &=\scalar{\Phi}{ (\oS 
+\oS^\dagger 
)\Phi}_{\M_\T}  \geq    c\scalar{\Phi}{ \Phi}_{\M_\T}\,,
\end{aligned}
\end{equation}
for some $c>0$, where we used condition (P) in Definition~\ref{def:symm syst}.
Since $\Phi\in\oB^\dagger|_{\M_\T}$, by definition of adjoint boundary space we 
have  
 $$  \scalar{\Phi}{  \sigma_\oS(\n)\Phi }_{\partial {\M_\T} }\leq 0 $$
Therefore,  \eqref{passaggio1} reduces to 
$$c\scalar{\Phi}{ \Phi}_{\M_\T} \leq {2} \scalar{\Phi}{ 
\oS^\dagger\Phi}_{\M_\T}$$
and by the Cauchy-Schwarz inequality we obtain the desired inequality.

Let us now assume furthermore that the bilinear forms $\fiber{\sigma_\oS (\n_0) \cdot}{\cdot}_p$ and $\fiber{\sigma_\oS (\n_\T) \cdot}{\cdot}_p$ are positive definite on
$\E_p$ for every $p\in\Sigma_0$ resp. $\Sigma_\T$.
Using again the Green identity and arguing as above we can conclude that
\begin{equation*}
\begin{aligned}
 -\scalar{\Psi}{  \sigma_\oS(\n^\flat)\Psi }_{\partial {\M_\T} }  +
{2} \scalar{\oS\Psi}{ \Psi}_{\M_\T}  &= \scalar{\Psi}{ (\oS 
+\oS^\dagger 
)\Psi}_{\M_\T}  \geq    c\scalar{\Psi}{ \Psi}_{\M_\T}\,,
\end{aligned}
\end{equation*}
for some $c>0$, where we used condition (P) in Definition~\ref{def:symm syst}. This time the sesquilinear form $\scalar{\Psi}{  \sigma_\oS(\n)\Psi }_{\partial {\M_\T} }$ is positive definite, since $\Psi\in\oB|_{\M_\T}$ and $\Psi|_{\Sigma_{0}}=0$ by assumption. \black
Using again the Cauchy-Schwarz inequality, we obtain the desired inequality.
\end{proof}

\section{$\mathbf L^{\mathbf 2}$-well-posedness in a time strip}
The aim of this section is to prove the $L^2$-well-posedness of the Cauchy 
problem for a Friedrichs system of constant characteristic {in a time strip 
$\M_\T:=t^{-1}((0,T))$ for $0<T\in\RR$}.
We shall achieve our goal in three steps:
first, we shall prove the existence and uniqueness of weak solutions.
 Second, we shall prove that any weak solution can be approximated by a 
sequence of smooth sections by means of a localization procedure.
 Finally, we shall discuss the regularity of strong solutions.
 
To this end, let $\|\cdot\|_{L^2(\E_{|_{\M_\T}})}$ be the norm corresponding 
to the scalar 
product~\eqref{eq:L2} and denote $L^2$-completion of $\Gamma_c(\E_{|_{\M_\T}})$ 
by 
\[L^2(\E_{|_{\M_\T}}) := 
\overline{\big({\Gamma_c(\E_{|_{\M_\T})}},{\scalar{.}{.}}_{{\M_\T}}  
\big)}^{\scalar{.}{.}_{{\M_\T}}} \,.\]

\subsection{Weak solutions}\label{sec:weak}

\begin{definition}\label{def:Weak existence}
We call $\Psi\in L^2(\E_{|_{\M_\T}})$ a \emph{weak solution} to the Cauchy 
problem~\eqref{CauchyM_T0} if the relation
\begin{equation*} 
\scalar{\Phi}{\f}_{\M_\T}=\scalar{\oS^\dagger \Phi}{\Psi}_{\M_\T}
\end{equation*}
holds for all  $\Phi \in \Gamma_{c}(\E|_{\M_\T})$ satisfying 
$\Phi|_{\Sigma_{0}} = 0$, $\Phi|_{\Sigma_{T}} = 0$ and $\Phi|_\bM \in 
\oB^\dagger|_{\M_\T}$.
\end{definition}

\begin{theorem}[Weak existence] \label{thm:weak}
Let $\M$ be a globally hyperbolic 
manifold with timelike boundary and let 
$t\colon\M\to\RR$ be a Cauchy temporal function.
For any $0<T\in\RR$ 
denote 
with $\M_\T:=t^{-1}(0,T)$ a time strip.
Let finally $\oS$ be a Friedrichs 
system and denote with $G_\oB$ a future admissible boundary condition.
Assume $\M$ to be Cauchy-compact when $\oS$ is symmetric 
hyperbolic. 
Then there exists a  weak solution $\Psi\in L^2(\E_{|_{\M_\T}})$ to the Cauchy 
problem~\eqref{CauchyM_T0} restricted to $\M_\T$.
\end{theorem}
\begin{proof}
By Theorem~\ref{thm:energyest}, we get that
for every $\Phi\in \Gamma_c(\E_{|_{\M_\T}})$ satisfying  
$\Phi|_{\Sigma_{0}}=0$, $\Phi|_{\Sigma_{T}}=0$ and $\Phi|_{\bM}\in 
\oB^\dagger|_{\M_\T}$ it holds
\begin{align}
\label{Energy Inequality2ter}
\|\Phi\|_{L^2(\E_{|_{\M_\T}})} \leq \tilde C \| \oS^\dagger 
\Phi\|_{L^2(\E_{|_{\M_\T}})}  \,.
\end{align}
The latter inequality implies that the kernel of the operator $\oS^\dagger$ 
acting on $$\text{dom}\,  \oS^\dagger:= \{ \Phi\in \Gamma_c(\E|_{\M_\T})\ |\  
\Phi|_{\Sigma_{0}}=0, \, \Phi|_{\Sigma_{T}}=0 ,\, 
\Phi|_\bM\in\oB^\dagger|_{\M_\T}\}$$ is trivial.  Let now $\ell\colon 
\oS^\dagger(\text{dom}\,  \oS^\dagger)\to \CC$ be the linear functional defined 
by
$$\ell(\Theta)=\scalar{\Phi}{\f}_{{\M_\T}}$$
where $\Phi$ satisfies $\oS^\dagger \Phi=\Theta$. By the energy 
inequality~\eqref{Energy Inequality2ter}, $\ell$ is bounded:
\begin{align*}
\ell(\Theta)=&\scalar{ \Phi}{ \f}_{\M_\T} \leq \| \f\|_{L^2(\E_{|_{\M_\T}})} \, 
\| \Phi\|_{L^2(\E_{|_{\M_\T}})} \\ \leq& \tilde C \| \f\|_{L^2(\E_{|_{\M_\T}})} 
\|\oS^\dagger  \Phi\|_{L^2(\E_{|_{\M_\T}})} = \tilde C \|  \f\|_\T 
\|\Theta\|_{L^2(\E_{|_{\M_\T}})} ,
\end{align*}
where in the first inequality we used the Cauchy-Schwarz inequality. 
Then $\ell$ can be extended to a continuous functional defined on the 
$L^2$-completion of $\oS^\dagger(\text{dom} \, \oS^\dagger)$ denoted by  
$\mathcal{H}\subset L^2(\E|_{\M_\T})$. Finally, by Riesz's representation 
theorem, 
there exists an element $  \Psi\in L^2(\E|_{\M_\T})$ such that 
$$ \ell({\Theta}) =  \scalar{\Theta}{  \Psi}_{\M_\T} \,. $$ 
for all $\Theta\in \oS^\dagger(\text{dom}\oS^\dagger)$. Thus, we obtain
$$ \scalar{  \Phi}{  \f}_{ {\M_\T}}= \ell({\Theta})=\scalar{\Theta}{  
\Psi}_{\M_\T} =\scalar{\oS^\dagger   \Phi}{  \Psi}_{\M_\T}$$
for all $  \Phi\in \text{dom} \oS^\dagger$, which proves the existence of a 
weak 
solution $ {\Psi}$. 
\end{proof}

Note that weak solutions from Definition \ref{def:Weak existence} have no reason to be unique since the Hilbert space $\mathcal{H}$ from the proof of Theorem \ref{thm:weak} does not coincide in general with $L^2(\E|_{\M_\T})$.\footnote{We thank Alexander Strohmaier for pointing a mistake out in the first version of the paper.}

\subsection{Strong solutions}\label{sec:strong sol}

\begin{definition} 
We call $\Psi\in L^2(\E_{|_{\M_\T}})$ a \emph{strong solution} of the 
initial-boundary value problem~\eqref{CauchyM_T0} if there exists a sequence 
of 
sections $\Psi_k \in \Gamma(\E_{|_{\M_\T}})\cap L^2(\E_{|_{\M_\T}})$ such that 
$\Psi_k{}_{|_\bM}\in\oB|_{\M_\T}$ on $\partial{\M_\T}$ and
$$ \| \Psi_k - \Psi \|_{L^2(\E_{|_{\M_\T}})} \xrightarrow{k\to \infty} 0 \quad 
\text{ and } \quad \| \oS\Psi_k - \f \|_{L^2(\E_{|_{\M_\T}})} \xrightarrow{k\to 
\infty} 0.$$ 
 \end{definition}
In order to show that any weak solution is a strong solution, we begin by 
localizing the problem.
Hence, consider an open covering $\{\U_j\}_j$ of 
$\M_\T$ and let 
$\varphi_j$ be a smooth partition of unity subordinated to $\U_j$.

\begin{lemma}\label{lem:local}
A section $\Psi \in L^2(\E_{|_{\M_\T}})$ is a weak solution of the Cauchy 
problem~\eqref{CauchyM_T0} if and only if for any $j$,  $\Psi_j:= \varphi_j 
\Psi$ is a weak solution of 
\begin{equation}\label{CauchyLocal}
\begin{cases}
 \oS\Psi_j=\f_j:=\varphi_j\f+ 
\sigma_\oS(d\varphi_j)\Psi \\
 \Psi_j|_{\Sigma_{0}}=\h_j:=\varphi_j\h \\
 \Psi_j|_\bM\in\oB|_{\M_\T}.
 \end{cases}
\end{equation}
\end{lemma}
\begin{proof}
To verify our claim, suppose  $\Psi$ satisfies $\oS \Psi = \f$ in a weak sense, 
\ie for any $\Phi\in\Gamma_c(\E_{|_{\M_\T}})$ satisfying 
$\Phi_{|_{\bM}}\in\oB^\dagger|_{\M_\T}$, and $\Phi_{|_{\Sigma_0}}=0$, it holds 
$ 
\scalar{\oS^\dagger 
\Phi}{\Psi}_\T=\scalar{\Phi}{\f}_\T\,.$
Using $\scalar{\Phi}{\varphi_j \Psi}_{\M_\T}=\scalar{\varphi_j 
\Phi}{\Psi}_{\M_\T}$ and then Leibniz rule, it follows that
\begin{align*}
\scalar{\oS^\dagger \Phi}{\Psi_j}_{\M_\T} &=  \scalar{\varphi_j\oS^\dagger 
\Phi}{\Psi}_{\M_\T} = \scalar{\oS^\dagger (\varphi_j\Phi)}{\Psi}_{\M_\T} - 
\scalar{(\oS^\dagger \varphi_j)\Phi}{\Psi}_{\M_\T} =\\
&=\scalar{\varphi_j \Phi}{\f}_{\M_\T} + \scalar{
\sigma_\oS(d\varphi_j)\Phi}{\Psi}_{\M_\T}=\scalar{ \Phi}{ \varphi_j \f + 
\sigma_\oS(d\varphi_j) \Psi }_{\M_\T} 
\end{align*}
This shows that $\Psi_j$ is a weak solution of the Cauchy 
problem~\ref{CauchyLocal}. Conversely, suppose that $\Psi_j$ is a weak solution 
of the Cauchy problem~\eqref{CauchyLocal}. Then by summing over $j$ and using 
$\sum_j d\varphi_j=0$, we find that a weak solution $\Psi=\sum_j\Psi_j$ is a 
weak solution of $\oS\Psi=\f$. 
\end{proof}

\begin{definition}
Let $\U\subset\M_\T$ be a compact subset in $\M_\T$.
We say that $\Psi\in L^2(\E_{|_\U})$ is a \emph{locally strong solution} of the 
Cauchy problem~\eqref{CauchyM_T0} if there exists a sequence of sections 
$\Psi_k 
\in \Gamma(\E_{|_\U})$ such that $\Psi_k{}_{|_\bM}\in\oB|_{\M_\T}$ on 
$\partial{\M_\T} 
\cap 
\U$ and  
$$ \| \Psi_k - \Psi \|_{L^2(\E_{|_\U})} \xrightarrow{k\to \infty} 0 \quad 
\text{ 
and } \quad \| \oS\Psi_k - \f \|_{L^2(\E_{|_\U})} \xrightarrow{k\to 
\infty} 0.$$ 
 \end{definition}

The strategy for the existence of strong solutions is the following.
First, it suffices to prove the existence of such solutions on compact subdomains of $\M_\T$; then considering an exhaustion of $\M_\T$ by compact subsets $(K_p)_{p\geq0}$ and associated cutoff functions $\chi_p\in C_c^\infty(\M_\T,[0,1])$ with $\chi_p{}_{|_{K_p}}=1$, the fact that each $\chi_p\Psi$ is a strong solution associated to $\chi_p\f$ on $K_p$ implies at the limit as $p\to\infty$ that $\Psi$ is a strong solution associated to $\f$ on $\M_\T$.
Second, each compact set can be covered by finitely many arbitrarily small open sets, therefore the existence of strong solutions can be localized to small coordinate neighbourhoods, which is what we do next.\\

We concentrate on points in the boundary $p\in  \partial \M\cap \M_\T$ (the
other points will even be easier to handle since we do not have to care about 
boundaries) 
and firstly define a convenient chart as follows, compare also 
Figure~\ref{fig:chart}:
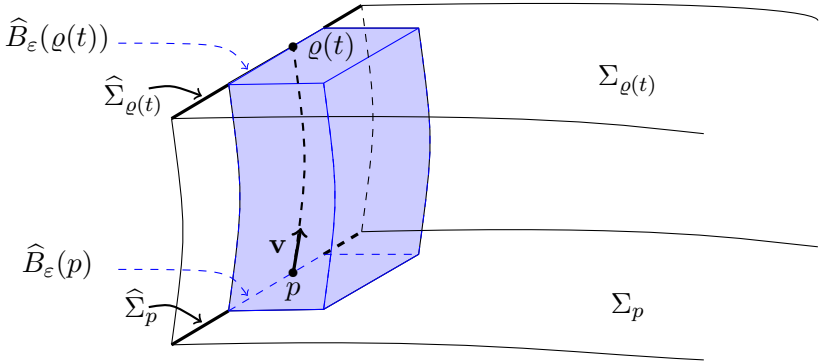
\begin{figure}[h!]
\begin{tikzpicture}

\filldraw[fill=blue!20!white] (0.75,0.45) --(2,0.47) -- (3.25,1.2) -- 
(3.25,1.2) 
.. controls (3.45,2.7) .. (3.25,4.2) -- (3.25,4.2) --(2,4.2) -- (0.75,3.47) .. 
controls (0.95,1.95) .. (0.75,0.45) ;

 \draw (0,0) .. controls (4,0.1) and  (5,-0.1) .. (7,-0.2) ; 
 \draw (2.5,1.5) .. controls  (6.5,1.6) and  (7.5,1.4) .. (8.5,1.3) ; 
\draw[very thick] (0,0) -- (0.75,0.45);
\draw[dashed, very thick] (2,1.2) -- (2.5,1.5);
\draw[blue,dashed ] (0.75,0.45) -- (2,1.2) ;
\draw[blue, ] (0.75,0.45) -- (2,0.47) ;
\draw[blue, dashed] (2,1.2) -- (3.25,1.2) ;
\draw[blue, ] (2,0.47) -- (3.25,1.2) ;

 \draw (0,3) .. controls  (5,3.1) and  (6,2.9)  .. (7,2.8) ;  
\draw (2.5,4.5)..controls (7.5,4.6) and (8.5,4.4) .. (8.5,4.3) ;  
\draw[very thick] (0,3) -- (0.75,3.45);
\draw[very thick] (2,4.2) -- (2.5,4.5);
\draw[blue,] (0.75,3.45) -- (2,4.2) ;
\draw[blue, ] (0.75,3.45) -- (2,3.47) ;
\draw[blue, ] (2,4.2) -- (3.25,4.2) ;
\draw[blue, ] (2,3.47) -- (3.25,4.2) ;

 \draw  (0,0) .. controls (0.2,1.5) .. (0,3); 
\draw[dashed]   (2.5,1.5) .. controls (2.7,3) .. (2.5,4.5); 

\draw[dashed,blue,]  (0.75,0.45) .. controls (0.95,1.95) .. (0.75,3.47); 
\draw[dashed,blue,]  (2,0.47) .. controls (2.2,1.97) .. (2,3.47); 
\draw[blue,]  (2,0.47) .. controls (2.2,1.97) .. (2,3.47); 
\draw[dashed,blue,]  (3.25,1.2) .. controls (3.45,2.7) .. (3.25,4.2); 

            \node  at (1.6,0.95) {$\mathsmaller{{\bullet}}$};

            \node  at (1.6,0.70) {$p$};       
            \node  at (1.6,3.95) {$\mathsmaller{{\bullet}}$};

            \node  at (2.1,3.95) {$\varrho(t)$};

\draw[dashed,black, thick]  (1.6,0.95)  .. controls (1.8,2.45) .. (1.6,3.95); 
\draw[black, very thick,->]  (1.6,0.95)  -- (1.7,1.55); 
            \node  at (1.4,1.3) {$\bf v$};       

            \node  at (6,0.5) {$\Sigma_p$};       
            \node  at (6,3.5) {$\Sigma_{\varrho(t)}$};

    \draw[dashed,blue, ,->]  (-0.7,1) .. controls (0.75,1) .. (1,0.65) ; 
     \node  at (-1.5,1.1) {$\widehat{B}_\epsilon(p)$};
   
    \draw[dashed,blue, ,->]  (-0.7,4) .. controls (0.75,4) .. (1,3.65) ; 
     \node  at (-1.5,4.1) {$\widehat{B}_\epsilon(\varrho(t))$};
                   
            \node  at (-0.4,0.5) {$\widehat{\Sigma}_p$};       
    \draw[black,thick ,->]  (-0.3,0.5)  .. controls (0,0.55) .. (0.4,0.3) ; 

            \node  at (-0.5,3.3) {$\widehat{\Sigma}_{\varrho(t)}$};       
   \draw[black, thick,->]  (-0.3,3.4)  .. controls (0,3.45) .. (0.4,3.3) ; 

    \end{tikzpicture}
    \caption{Fermi coordinates on each Cauchy surface.}
    \label{fig:chart}
\end{figure} 
 Let ${\Sigma}_p$ be the Cauchy surface of ${\M_\T}$ to which $p$ belongs to. 
For $q\in \partial {\M}\cap\M_\T$ let $\widehat{\Sigma}_q := \Sigma_q\cap 
\partial \M \cap \M_\T$ be the corresponding Cauchy surface in the boundary. 
 Let $\varrho\colon [0,T]\to \partial \M\cap\M_\T$ be the timelike 
geodesic in $\partial \M\cap\M_\T$ starting at $p$ with velocity ${\bf v} \in 
\T_p\bM$ where ${\bf v}$ is a normalized, future-directed, timelike vector 
perpendicular to $\widehat{\Sigma}_p$ in $\partial \M \cap\M_\T$. 
Let 
$\widehat{B}_\epsilon (\varrho(t))$ be the $\epsilon$-ball in 
$\partial\hat{\Sigma}_{\varrho(t)}$ around $\varrho(t)$. On these balls we 
choose 
geodesic normal coordinates $\widehat{\kappa}_t\colon 
B_\epsilon^{n-1}(0)\subset 
\mathbb R^{n-1}\to \widehat{B}_\epsilon (\varrho(t))$. Moreover, inside each 
$\Sigma_{\varrho(t)}$ we choose Fermi coordinates with base 
$\widehat{B}_\epsilon (\varrho(t))$.
  Thus, we obtain a chart in $\Sigma_{\varrho(t)}$ around $\varrho(t)$ as 
\label{def:U_p}  
\begin{eqnarray*}
\widetilde\kappa_t\colon B_\epsilon^{n-1}(0)\times [0,\epsilon]\subset \mathbb 
R^{n}&\to& U_{\epsilon} (\widehat{B}_\epsilon (\varrho(t)))
\\
 (y,z) & \mapsto& \exp^{\perp, \Sigma_{\varrho(t)}}_{\widehat{\kappa}_t(y)}(z)
\end{eqnarray*}
where $U_{\epsilon} (\widehat{B}_\epsilon (\varrho(t))):= \{ q\in 
\Sigma_{\varrho(t)}\ |\ \mathrm{dist}_{\Sigma_{\varrho(t)}}(q, 
\widehat{B}_\epsilon (\varrho(t))\leq \epsilon\}$, $\exp^{\perp, \Sigma_{\varrho(t)}}_{\widehat{\kappa}_t(y)}(z)$ is the 
normal exponential map in $\Sigma_{\varrho(t)}$ starting at 
$\widehat{\kappa}_t(y)$ with velocity perpendicular to 
$\widehat{\Sigma}_{\varrho(t)}=\partial \Sigma_{\varrho(t)}$ pointing in the 
interior and with magnitude $z$. Putting all this together we obtain a 
chart

$$  \begin{aligned}
   \kappa_{p} \colon [0,T] \times B_\epsilon^{n-1}(0)\times [0,\epsilon] 
\subset \mathbb R^{n+1}&\to \U_p:= \bigcup_{t\in [0,\epsilon]} U_\epsilon 
(\widehat{B}_\epsilon (\varrho(t))) \subset \M_\T \\
   (t,y,\bar{z})&\mapsto \widetilde\kappa_t(y,\bar{z}).
  \end{aligned}
$$

For us here, the only purpose of those charts is to specify coordinates such 
that near the point $p$ the Cauchy problem is close enough to the 
Minkowski standard form.   

Next we prove that weak solutions are strong solutions.
     
\begin{proposition}\label{prop:weak-strong}
Let $\M$ be a globally hyperbolic 
manifold with timelike boundary and let 
$t\colon\M\to\RR$ be a Cauchy temporal function.
For any $0<T\in\RR$ 
denote 
with $\M_\T:=t^{-1}(0,T)$ a time strip.
Let finally $\oS$ be a Friedrichs 
system and denote with $G_\oB$ a future admissible boundary condition.
Assume $\M$ to be Cauchy-compact when $\oS$ is symmetric 
hyperbolic. 
Then any  weak solution $\Psi\in L^2(\E_{|_{\M_\T}})$ to the Cauchy  problem~\eqref{CauchyM_T0} restricted to $\M_\T$ is a strong solution.
\end{proposition}
\begin{proof}
On account 
of the above discussion, it is enough to check that any weak solution in the coordinate neighbooorhood chosen above is a strong solution.
For every such coordinate neighbourhood $\kappa_p$, a symmetric positive system $\oS$
has the form
$$  \oS= \sigma_\oS(d  t) \partial_{ t} + \sum_{j=1}^{n-1} 
\sigma_\oS(dx^j)\partial_{x^j} + \sigma_\oS(d\overline{z}) 
\partial_{\overline{z}} + C( t, y, \overline z)
 $$   
for some zero-order operator $C$.
Therefore, in case the coordinate patch meets $\bM_\T$, \cite[Theorem 8]{Rauch} ensures that
any weak solution is a strong solution.
In case the coordinate patch does not meet $\bM_\T$, then using
the classical results of Friedrichs~\cite{weak=strong} we can conclude.
Recall that, by Proposition 
\ref{Cauchycompactornotisthesameforsymmhyp}, if $\oS$ is a symmetric hyperbolic 
system, then $\M$ may be assumed to be Cauchy-compact. 
\end{proof}

We are finally in the position to prove our first main result.

\begin{proof}[Proof of Theorem~\ref{thm:main1}.]
To prove our claim it remains to show that,  when  $\scalar{\Psi}{  \sigma_\oS(\n)\Psi }_{\partial {\M_\T} }$ is positive definite, then any strong solutions are unique.
But this is a direct consequence of the energy estimate (see Theorem~\ref{Energy Inequality2}). Indeed, suppose that there exists two strong solutions $\Psi_1$ and $\Psi_2$.
Then their difference $\Psi:=\Psi_1-\Psi_2$ solves the homogeneous Cauchy problem, i.e. $\f=0$ and $\Psi |_{\Sigma_0}=0$. This implies that
$$\|\Psi\|_{L^2(\E|_{\M_\T})} \leq \tilde{D} \|\oS\Psi\|_{L^2(\E|_{\M_\T})} = \tilde{D}\|f\|_{L^2(\E|_{\M_\T})} =0\,.$$
This ends the proof of Theorem~\ref{thm:main1}.\black
\end{proof}

\subsection{Differentiability of the solutions for symmetric hyperbolic 
systems}\label{sec:regularity}
It is well-known that the Cauchy problem for the backward heat 
equation is not well-posed.
This is because an initial data for the backward 
heat equation is a final condition for the forward heat equation.
The latter 
equation has a smoothening effect on the initial data, \ie the solution is 
smooth even if the initial data is only continuous.
It is easy to understand 
that, there exists a class of smooth initial data for the  backward heat 
equation generating non-smooth solutions.
Since the heat equation can be 
reduced to a symmetric positive system (\cf 
Section~\ref{sec:reactiondiffusion}), we cannot expect the existence of smooth 
solutions for a generic symmetric positive system.
Hence, in this section we 
shall only focus 
 only on the subclass of symmetric hyperbolic systems.
In particular, we shall see that, if the Cauchy data $(\f,\h)$ are 
smooth and a compatibility condition is imposed, then the strong solution is 
actually smooth. To this end, let
 $t\colon\M\to\RR$ be a Cauchy temporal function with 
gradient tangent to the boundary, as in Theorem~\ref{thm: Sanchez}, and write 
the symmetric hyperbolic system $\oS$ as
$$\oS=\sigma_{\oS}(dt)\nabla_{t} - \H \,$$
where $\H$ is a first-order linear differential operator which differentiates 
only in the directions that are tangent to $\Sigma$ and where $\nabla$ 
is 
 any fixed metric connection for
$\fiber{\cdot}{\cdot}$. 
Finally  let $G_{\oB_+},G_{\oB_-}\colon 
\E_{|_{\bM}}\longrightarrow\E_{|_{\bM}}$ be future and past
admissible boundary conditions for $\oS$, in 
particular $\oB_\pm:=\ker(G_{\oB_\pm})$ defines the future resp. past
admissible boundary space for $\oS$ along $\bM$.
The compatibility conditions of order $k\geq0$ for 
$\h\in\Gamma(\E_{|_{\Sigma_{0}}})$ and  
$\f\in\Gamma(\E)$ read
\begin{equation}\label{eq:comp cond data}
 \sum_{j=0}^k 
\frac{(k)!}{j! (k-j)!}
\Big( \nabla_t^j  G_{\oB_+} \Big)\Big|_{\partial\Sigma_0} \h_{k-j} =0,
\end{equation}
and 
\begin{equation}\label{eq:comp cond databis}
 \sum_{j=0}^k 
\frac{(k)!}{j! (k-j)!}
\Big( \nabla_t^j  G_{\oB_-} \Big)\Big|_{\partial\Sigma_0} \h_{k-j} =0,
\end{equation}
\color{black}
where the sequence $(\h_k)_k$ of sections of $E_{|_{\partial\Sigma_0}}$ is 
defined inductively by $\h_0:=\h$ and
$$ \h_k:= 
\sum_{j=0}^{k-1} 
\frac{(k-1)!}{j! (k-1-j)!}
 \H_j{}_{|_{\partial\Sigma_0}} \, \h_{k-1-j} + 
\nabla_t^{k-1} \big(\sigma_{\oS}^{-1}(dt)\f)_{|_{\partial\Sigma_0}} \qquad 
\text{for all }k\geq1,
$$  
where $\H_j:=[\nabla_t,\H_{j-1}]$ and $\H_0:=\sigma_{\oS}(dt)^{-1}\H$.
\begin{notation}\label{not:comp cond data}
We denote the space of data which satisfy the compatibility conditions as 

$$\overline{\Gamma(\E_{|_{\M_\T}})\times{\Gamma(\E_{|_{\Sigma_{0}}})}}:= \{ 
(\f,\h) \in \Gamma(\E_{|_{\M_\T}})\times{\Gamma(\E_{|_{\Sigma_{0}}}}) \;| \; 
\textrm{ \eqref{eq:comp cond data} 
and~\eqref{eq:comp cond databis} hold}\;\}\,.  
$$ 
\end{notation}

The compatibility conditions~\eqref{eq:comp cond data} and \eqref{eq:comp cond 
databis} up to order $k$ must be 
fulfilled in order for the solution of the Cauchy-problem, if it exists, to be 
$C^k$. 
Those conditions are sufficient for nowhere characteristic  symmetric 
hyperbolic  systems \cite{Rauch-Massey}.
However, if the symmetric hyperbolic
system is of nonvanishing constant characteristic, full regularity of the 
solution cannot be expected in general, see e.g. \cite{Tsuji}.
In that case, as 
shown 
in \cite{Rauch}, there exists a good notion of tangential regularity.
Given the Cauchy hypersurface $\Sigma_0\subset\M$ where the initial condition 
is fixed, we say 
that a vector field $X\in\Gamma(\T\Sigma_0)$ is tangential to 
$\partial\Sigma_0$ if and only if 
$X_{|_{\partial\Sigma_0}}\in\Gamma(\T\partial\Sigma_0)$, i.e. $g(X,\n)_q=0$ 
for every $q\in\partial\Sigma_0$.
We denote the space of tangential vector 
fields as 
$$\mathfrak{X}_{\rm tan}(\Sigma_0):=\{ X\in\Gamma(\T\Sigma_0) \,|\, g(X,\n)_q=0 
\text{ for 
all $q\in\partial\Sigma_0$} \}\,.$$
As in \cite{SecchiIBVP96}, we define -- at least in the case where 
$\M$ is Cauchy-compact, otherwise a metric along $\Sigma_0$ has to be fixed 
-- for any $m\geq0$ the anisotropic Sobolev space $H_*^m(\E_{|_{\Sigma_0}})$ 
as 
\begin{eqnarray*}
H_*^m(\E_{|_{\Sigma_0}})&:=&\Big\{\phi\in\L^2(\E_{|_{\Sigma_0}}),\,\nabla_{X_1}
\cdots\nabla_{X_h}\nabla_{X_1'}\cdots\nabla_{X_k'}\phi\in\L^2(\E_{|_{\Sigma_0}}
)\\
&&\phantom{\Big\{}\forall\,X_1,\ldots,X_h,X_1',\ldots,X_k'\Big\},
\end{eqnarray*}
where $X_1,\ldots,X_h,X_1',\ldots,X_k'$ are smooth tangent vector fields on 
$\Sigma_0$ with $X_1,\ldots,X_h\in\mathfrak{X}_{\rm 
tan}(\Sigma_0)$, $X_1',\ldots,X_k'\notin\mathfrak{X}_{\rm 
tan}(\Sigma_0)$ as well as $h+2k\leq m$.
The space $H_*^m(\E_{|_{\Sigma_0}})$ can be endowed with a Hilbert-space 
structure, see \cite[p. 673]{SecchiIBVP96}.
It is easy to see that $H_*^m(\E_{|_{\Sigma_0}})\subset 
H^{[\frac{m}{2}]}(\E_{|_{\Sigma_0}})$, in particular $H_*^m(\E_{|_{\Sigma_0}})$ 
embeds continuously into the space $\Gamma^p(\E_{|_{\Sigma_0}})$ of $C^p$ 
sections of $\E_{|_{\Sigma_0}}$ as soon as $m>\frac{n}{2}+p$ by the Sobolev 
embedding theorem for compact manifolds with $C^1$ boundary, see e.g. 
\cite[Ch. V]{Adamsbook}.
For the sake of completeness, we recall part of Secchi's main result 
\cite[Theorem 2.1]{SecchiIBVP96} in our context: Fix any integers
$m\geq2[\frac{n}{2}]+6$ and $1\leq s\leq m$.
Assume $\Sigma_0$ to be compact 
with smooth boundary and of nonzero and nonmaximal constant 
characteristic w.r.t. a symmetric hyperbolic system $\oS$ on $\M$.
Let $\G_\oB$ to 
be a future admissible boundary condition for $\oS$ along $\bM$.
Given $\f\in\bigcap_{j=0}^s H^j([0,\T],H_*^{s-j}(\Sigma_0))$ and $\h\in 
H_*^s(\Omega)$, assume that the compatibility conditions \eqref{eq:comp cond 
data} are satisfied up to order $s-1$ and that $\h_j\in H_*^{s-j}(\Sigma_0)$ 
for all $j=0,\ldots,s-1$.
Then there exists a unique $\Psi\in\bigcap_{j=0}^s 
\Gamma^j([0,\T],H_*^{s-j}(\Sigma_0))$ solution of the initial boundary value 
problem $\oS\Psi=\f$ on $\M_\T$, $\Psi_{|_{\Sigma_0}}=\h$ and 
$\Psi_{|_\bM}\in\oB|_{\M_\T}$.
Note in particular that, if $\f,\h$ are smooth on $\M$, then so must be $\Psi$ 
because it lies in $\Gamma^j\left([0,T],\Gamma^p(\E_{|_{\Sigma_0}})\right)$ for 
any $j,p$.

\begin{theorem}\label{thm:tangential regularity}
Let $\M$ be a globally hyperbolic manifold with timelike boundary and let $\oS$ 
be a 
symmetric hyperbolic system of constant characteristic.
 If the data 
$(\f,\h)$ are smooth and satisfy the compatibility 
conditions~\eqref{eq:comp cond data} and \eqref{eq:comp cond databis}, then  
the strong solution $\Psi$ of the Cauchy problem~\eqref{CauchyM_T0} lies in 
$\Gamma(\E_{|_{\M_\T}})$.
 \end{theorem}
\begin{proof}
First let $p\in \partial \M\cap \Sigma_{\T}$ for some 
 and let 
$\varrho\colon [0,\T]\to \partial \M$ be a 
timelike curve
with $\varrho(0)\in \Sigma_{0}$ and $p=\varrho(\T)$. 
We fix $\epsilon>0$ such that 
we have Fermi coordinates on a ``cube'' $U_{\varrho(t)}$ around $\varrho(t)$
as in Section~\ref{sec:strong sol} for all $t\in [0,\T]$. This is always 
possible since  the image of
$\varrho$ is compact and everything depends smoothly on the base point.
For $\displaystyle{\tilde{U}_p:=\bigcup_{t\in[0,\T]}U_{\varrho(t)}}$ we know 
that the compatibility conditions~
\eqref{eq:comp cond data} and \eqref{eq:comp cond databis} are fulfilled along 
$\Sigma_0$ by assumption.
Thus, 
\cite[Theorem 2.1]{SecchiIBVP96} tells us that the strong solution $\Psi$ lies 
in 
$\Gamma(\E_{|_{\tilde{U}_p}})$.
If $\oS$ is nowhere characteristic, $\Psi$ is actually smooth on account of 
\cite[Theorem 3.1]{Rauch-Massey}.
For $p\in \M\setminus \partial \M$ we choose a timelike curve $\varrho
\colon [0,\T]\to \M\setminus \partial \M$  with $\varrho(0)\in 
\Sigma_{0}$ and $p=\varrho(\T)$ and proceed as before.
It is even easier since we can just use geodesic normal coordinates in the 
Cauchy hypersurfaces around 
each $\rho(t)$.
\end{proof}

\section{Global well-posedness}\label{sec:global well-posedness}

Up to now we have obtained a strong solution in any time strip $\M_\T$  in 
Theorem~\ref{thm:main1} and showed that, if the  Cauchy data $(\f,\h)$ fulfill
the compatibility condition~\eqref{eq:comp cond data}, then the solution is 
actually smooth (\cf 
Theorem~\ref{thm:tangential regularity}). 
We can now easily put everything together to  obtain global well-posedness of 
the Cauchy problem for a symmetric hyperbolic system of constant 
characteristic. 
\begin{proof}[Proof of Theorem~\ref{thm:main2}]
Fix $\h\in\Gamma_{c}(\E_{|_{\Sigma_0}})$. 
On account of Theorem~\ref{thm:weak}, for any $T\in [0,\infty)$ there exists a 
weak solution $\Psi_T$ to the Cauchy problem~\eqref{CauchyM_T0} in the time 
strip $\M_T:= t^{-1}([0,T])$. By Theorem~\ref{thm:tangential regularity}, we 
get 
in particular that $\Psi_T$ is smooth in the time strip $\M_T$. By uniqueness 
of 
the solution, we get 
$\Psi_{T_1}|_{t^{-1}[0,T_1]}=\Psi_{T_2}|_{t^{-1}[0,T_1]}$ for all $T_1\leq 
T_2\in [0,\infty)$. 
{By combining a similar argument for negative time with 
Lemma~\ref{l:timereversalSsymmhyp}, we get existence of solutions for negative 
times.} \\
Finally, the stability of the Cauchy problem follows by~\cite[Section 5]{Ba}, 
the fact that we have a boundary condition playing no role in the 
proof.
\end{proof}

A byproduct of the well-posedness of the Cauchy problem is the existence of 
Green operators:
 
\begin{proposition}\label{prop:Green}
A symmetric hyperbolic system of constant characteristic on a globally 
hyperbolic manifold with timelike boundary coupled with an admissible boundary 
condition is Green-hyperbolic, i.e., 
there exist linear maps, called 
\textit{advanced/retarded Green operator} respectively,
$\G^\pm\colon \Gamma_{c}(\E)  \to \Gamma_{sc,\oB_\pm}(\E)$ satisfying
\begin{enumerate}[(i)]
\item
$\oS\circ \G^\pm \f=\f$ for all $ \f \in 
\Gamma_{c}(\E)$ and $\G^\pm\circ\oS\f=\f$ for all $ \f \in 
\Gamma_{c,\oB_\pm}(\E)$;
\item  $\supp(\G^\pm \f ) \subset J^\pm (\supp \f )$ for all 
$\f\in\Gamma_{c}(\E)\,,$
\end{enumerate}
where $J^\pm$ denote the causal future (+) and past ($-$) and 
$\Gamma_{\sharp,\oB_\pm}(\E)\subset\Gamma_{\sharp}(\E)$, $\sharp\in\{sc,c\}$ 
denotes the space of smooth  sections on $\E$ (with $\sharp$ support property) 
which fulfill the $\oB_\pm$-boundary condition.
\end{proposition}

\begin{proof}
Let $\f\in \Gamma_{c}(\E)$.
We choose $t_0\in \RR$ such that $\supp \f \subset J^+( 
\Sigma_{t_0})$.
By Theorem~\ref{thm:main2}, there exists a unique solution $\Psi=\Psi(\f)$ to 
the 
Cauchy problem 
$$
\begin{cases}
\oS \Psi =  \f\\
\Psi_{|_{\Sigma_{t_0}}} = 0 \\
\Psi_{|_{\bM}} \in\oB_+.
\end{cases}
$$
We set  
$\G^+\f:=\Psi$ and notice that $\oS\circ\G^+\f= \oS \Psi = \f$. Note that by 
the finite speed of propagation, (cf. Proposition~\ref{prop:finite}), 
$\G^+\f\in \Gamma_{sc,\oB_+}(\E)$. 
Moreover, $\G^+ \circ \oS \Psi= \G^+\f=\Psi$ 
which shows $(i)$.
By Proposition~\ref{prop:finite}, we obtain 
$\supp\G^+\f \subset J^+(\supp f)$ and this concludes the proof of $(ii)$ for 
$\G^+$.
The existence of the retarded Green operator $\G^-$ is proven analogously.
\end{proof}

\section{Examples of symmetric hyperbolic systems}\label{sec:Ex SHS}

\subsection{The Euler momentum equation}\label{sec:Euler}

We briefly discuss an elementary example of a symmetric hyperbolic system where 
the notion of future and past admissible boundary conditions is essential.
Most of what we present here is inspired from \cite[Example 
p.305]{Rauch-Massey}, \cite[Exercise 3.9.22]{Ba-lect} and can be found in 
\cite{dragoginouxmurro2022}.\\

We consider $\M:=\RR\times\Sigma$ with metric $g:=-dt^2\oplus h$, where 
$(\Sigma^n,h)$ is an arbitrary Riemannian manifold with nonempty boundary.
Let $\E:=\pi_2^*\T\Sigma\to\M$ be the tangent bundle of $\Sigma$ pulled back 
onto $\M$ via the second canonical projection $\pi_2\colon\M\to\Sigma$.
Given any section $u_0$ of $\E$, that is, any possibly time-dependent smooth 
vector field along $\Sigma$, we define the so-called linearized Euler operator 
\[\oS:=\partial_t+\nabla_{u_0}+\nabla_{\cdot}u_0\]
acting on sections of $\E$.
Because of $\sigma_{\oS}(\xi)=\xi(\partial_t+u_0)\cdot\mathrm{Id}$ for every 
$\xi\in \T^*\M$, the operator $\oS$ is symmetric and is hyperbolic if and only 
if $h(u_0,u_0)\leq1$ everywhere on $\M$.
Denoting by $\n$ the outward unit normal to $\bM$, we have 
$\sigma_{\oS}(\n^\flat)=h(\n,u_0)\cdot\mathrm{Id}=\langle\n,
u_0\rangle\cdot\mathrm { Id } $ , so that $\oS$ is of constant characteristic 
as 
soon as $\langle\n,u_0\rangle$ vanishes either identically or nowhere along 
$\bM$:
\begin{enumerate}
\item If $\langle\n,u_0\rangle>0$ along $\bM$, then $\oS$ is nowhere 
characteristic along $\bM$ and the only possible future and past admissible 
conditions for $\oS$ along $\bM$ are $\oB_+=\E_{|_{\bM}}$ and $\oB_-=\{0\}$.
\item If $\langle\n,u_0\rangle=0$ along $\bM$, then $\oS$ is of constant 
characteristic and the only possible future and 
past admissible 
conditions for $\oS$ along $\bM$ are $\oB_\pm=\E_{|_{\bM}}$.
\item If $\langle\n,u_0\rangle<0$ along $\bM$, then $\oS$ is nowhere 
characteristic along $\bM$ and the only possible future and past admissible 
conditions for $\oS$ along $\bM$ are $\oB_+=\{0\}$ and $\oB_-=\E_{|_{\bM}}$.
\end{enumerate}

\noindent From now on, we assume 
$\Sigma:=\RR_+^n=\{x=(x_1,\cdots,x_n)\in\RR^n\,|\,x_n\geq0\}$ with standard 
Euclidean metric and $u_0$ to be the restriction of any nonzero parallel vector 
field from $\RR^n$ to $\RR_+^n$.
Up to rescaling $u_0$, we may assume that $h(u_0,u_0)\leq1$ on $\RR_+^n$.
We fix $\Sigma_0:=\{0\}\times\Sigma$ as a Cauchy hypersurface in $\M$ and 
$v_0\in\Gamma(\E_{|_{\Sigma_0}})$ as initial data.
Consider the Cauchy problem for $\oS$:
\begin{equation}\label{eq:ibvpEuler}
\left\{\begin{array}{lll}\oS u&=0&\textrm{ on }\M\\ 
 u_{|_{\Sigma_0}}&=v_0&\textrm{ on }\Sigma_0\\ 
u_{|_{J^+(\Sigma_0)\cap\bM}}&\in\oB_+&\textrm{ along }J^+(\Sigma_0)\cap\bM\\ 
u_{|_{J^-(\Sigma_0)\cap\bM}}&\in\oB_-&\textrm{ along 
}J^-(\Sigma_0)\cap\bM\end{array}\right.
\end{equation}
The equation $\oS u=0$ on $\M$ with initial data $u_{|_{\Sigma_0}}=v_0$ 
along $\Sigma_0$ has a unique solution $u$ which can be explicitely written as
\[u(t,x)=v_0(x-tu_0)\]
for all $(t,x)\in\M$.
Clearly, if $\langle\n,u_0\rangle>0$ along $\bM$ (which is the case as soon as 
this i\-ne\-qua\-li\-ty is satisfied at one point of $\bM$), then no boundary 
condition 
can be imposed for $u$ along $J^+(\Sigma_0)\cap\bM$, whereas $u$ must vanish 
along $J^-(\Sigma_0)\cap\bM$, otherwise there would exist infinitely many 
solutions to \eqref{eq:ibvpEuler}.
This is precisely what the boundary conditions $\oB_\pm$ prescribe in that case.
Analogously, if $\langle\n,u_0\rangle<0$ along $\bM$, then no boundary 
condition can 
be imposed for $u$ along $J^-(\Sigma_0)\cap\bM$, whereas $u$ must vanish along 
$J^+(\Sigma_0)\cap\bM$, otherwise the same violation of uniqueness for 
solutions to symmetric hyperbolic systems occurs.
If $\langle\n,u_0\rangle=0$ along $\bM$, then no boundary condition at all, 
whether in the past or the future of $\Sigma_0$, can be imposed, which is 
consistent with the fact that the curves $t\mapsto x-tu_0$ in that case run 
either entirely along $\bM$ or in $\M\setminus\bM$.\\

In all three cases, the compatibility conditions \eqref{eq:comp cond data} and 
\eqref{eq:comp cond databis} for the solution $u$ only mean that $v_0$ vanishes 
along $\partial\Sigma$ as well as all its time derivatives.
\color{black}

\subsection{The classical Dirac operator}\label{sec:spin geom}

Let $(\M,g)$ be a globally hyperbolic manifold of dimension $n+1$ and assume to 
have a spin 
structure 
\ie  a twofold covering map from the $\Spin_0(1,n)$-principal bundle $\P_{\rm 
Spin_0}$  to the bundle of positively-oriented tangent frames $\P_{\rm SO^+}$ 
of 
$\M$ such that the following diagram is commutative:
\begin{flalign*}
\xymatrix{
\P_{\Spin_0} \times \Spin_0(1,n) \ar[d]_-{} \ar[rr]^-{} && \P_\Spin \ar[d]_-{} 
\ar[drr]^-{}   \\
\P_{\rm SO^+} \times \textnormal{SO}(1,n)   \ar[rr]^-{} &&  \P_{\rm SO^+}  
\ar[rr]^-{} &&  \M
}
\end{flalign*}
The existence of spin structures is related to the topology of $\M$.
A 
sufficient (but not necessary) condition for the existence of a spin structure 
is the parallelizability of the manifold.
Therefore, since any $3$-dimensional 
orientable manifold is parallelizable, it follows by Theorem~\ref{thm: Sanchez} 
that any 4-dimensional globally hyperbolic manifold admits a spin structure.
Given a fixed spin structure, one can use the spinor representation to 
construct 
the {spinor bundle}, \ie the complex vector bundle
$$\S\M:=\Spin_0(1,n)\times_\rho \CC^N$$
where $\rho: \Spin_0(1,n) \to \textnormal{Aut}(\CC^N)$ is the complex 
$\Spin_0(1,n)$ representation and $N:= 2^{\lfloor \frac{n+1}{2}\rfloor}$. 
The spinor bundle comes together with the following structures:
\begin{itemize}
\item[-] a natural $\Spin_0(1, n)$-invariant indefinite fiber metric
\begin{equation*}\label{eq: spin prod}
\fiber{\cdot}{\cdot}_p: \S_p\M \times \S_p\M \to \CC;
\end{equation*}
\item[-] a \textit{Clifford multiplication}, \ie  a fiber-preserving map 
$$\gamma\colon \T\M\to \text{End}(\S\M)$$ 
which satisfies
 for all $p \in \M$, $u, v \in \T_p\M$ and $\psi,\phi\in \S_p\M$
\begin{equation*}\label{eq:gamma symm}
 \gamma(u)\gamma(v) + \gamma(v)\gamma(u) = -2g(u, v)\Id_{\S_p\M}\,,\,  \fiber{\gamma(u)\psi}{\phi}_p=\fiber{\psi}{\gamma(u)\phi}_p\,.
\end{equation*}
\end{itemize}
\begin{definition} 
The \textit{(classical) Dirac operator} $\Dir$ is the operator defined as the 
composition of the metric connection $\nabla^\S$ on $\S\M$, obtained as a lift 
of the Levi-Civita connection on $\T\M$, and the Clifford multiplication:
$$\Dir=\gamma\circ\nabla^{\S\M} \colon \Gamma(\S\M) \to \Gamma(\S\M)\,.$$
\end{definition}

In local coordinates and with a trivialization of the spinor bundle $\S\M$, the 
Dirac operator reads as
\begin{align*}
\Dir \psi = \sum_{\mu=0}^{n}  \epsilon_\mu \gamma(e_\mu) \nabla^{\S\M}_{e_\mu} 
\psi\, 
\end{align*}
where  $\{e_\mu\}$ is a local Lorentzian-orthonormal frame of the tangent bundle $\T\M$ and 
$\epsilon_\mu=g(e_\mu,e_\mu)=\pm 1$.

\begin{proposition} The classical Dirac operator $\Dir$ on globally hyperbolic 
spin manifolds $\M$ with timelike boundary is a nowhere characteristic 
symmetric 
hyperbolic system.
\end{proposition}
\begin{proof}
Our claim follows from~\cite[Proposition 2.15]{defarg} and {\cite[Corollary 
3.12]{DiracMIT}}.
\end{proof}

\medskip

{\bf\large Examples of admissible boundary conditions}\\

The aim of this section is to test whether particular known boundary conditions 
for the Dirac 
operator are admissible in the sense of Definition \ref{def:admissible bc}. 
In particular, we shall show that the Lorentzian counterpart of the standard 
Riemannian boundary conditions are admissible, see e.g. \cite[Section 
1.5]{NicolasDirac}.

\paragraph{Lorentzian chirality boundary conditions.}\label{par:CBC}

Given a so-called chirality operator $\mathcal{G}$ on 
$\mathbb{S}\M$, i.e. a parallel involutive antiunitary (with respect to 
$\fiber{\cdot}{\cdot}$) endomorphism-field of $\mathbb{S}\M$ that anti-commutes 
with Clifford multiplication by  vectors, one 
may 
define the so-called \emph{chirality} boundary space which is defined as 
the range of the map
\[\pi_{\rm 
CHI}:=\frac{1}{2}\left(\mathrm{Id}-\gamma(\n)\mathcal{G}\right)\]
where $\gamma(\n)$ denotes Clifford multiplication for the outward-pointing 
unit 
normal along $\bM$.\\
The map $\pi_{\rm CHI}$ is clearly a linear projection since it satisfies 
$\pi_{\rm CHI}^2=\pi_{\rm CHI}$.
Furthermore, the range of $\pi_{\rm CHI}$ is the pointwise eigenspace of 
$\gamma(\n)\mathcal{G}$ to the eigenvalue $-1$ and 
$\mathcal{G}$ exchanges that eigenspace with the eigenspace to the eigenvalue 
$1$ since $\{\mathcal{G},\gamma(\n)\mathcal{G}\}=0$.
Therefore, the range of $\pi_{\rm CHI}$ has dimension 
$2^{\left[\frac{n}{2}\right]-1}$, which is the number of nonnegative -- 
actually positive -- 
eigenvalues of the endomorphism $\sigma_{\Dir}(\n^\flat)$.
Since $\mathcal{G}$ is skew-Hermitian with respect to the indefinite spin 
product $\fiber{\cdot}{\cdot}$, 
the complex number $\fiber{\mathcal{G}\psi}{\psi}$ must be 
imaginary for any $\psi\in\mathbb{S}\M_{|_{\bM}}$, therefore we have, for 
any 
$\psi\in\mathbb{S}\M_{|_{\bM}}$,
\begin{eqnarray*}
\fiber{\sigma_{\Dir}(\n^\flat)\pi_{\rm 
CHI}\psi}{\pi_{\rm 
CHI}\psi} &=&\fiber{\gamma(\n)\pi_{\rm 
CHI}\psi}{\pi_{\rm 
CHI}\psi}\\
&=&\fiber{\mathcal{G}\pi_{\rm 
CHI}\psi}{\pi_{\rm 
CHI}\psi},
\end{eqnarray*}
whose right-hand side is simultaneously real and imaginary and hence must 
vanish.
This proves the chirality condition to be admissible. Analogous arguments show 
that the range of $\pi_{\rm 
CHI}:=\frac{1}{2}\left(\mathrm{
Id} + \gamma(\n)\mathcal{G}\right)$ is also an admissible boundary space.

\begin{example}
An important example of a chirality operator is given by 
\[\mathcal{G}:=i^{\left[\frac{n}{2}\right]}\gamma(e_0) 
\gamma(e_1)\ldots\gamma(e_n) \colon\mathbb{S}\M\longrightarrow\mathbb{S}\M,\]
where $(e_0,e_1,\ldots,e_n)$ is any pointwise Lorentzian orthonormal basis of 
$\T\M$. 
Up to an imaginary scalar factor, $\mathcal{G}$ is the action of the volume 
form of $(\M,g)$.
It is easy to see that $\mathcal{G}$ is involutive and parallel and that, if 
$n$ 
is odd (i.e, $\M$ has even dimension), then $\mathcal{G}$ is skew-Hermitian 
(hence antiunitary)  with respect to $\fiber{\cdot}{\cdot}$ and anti-commutes 
with 
the Clifford action of any tangent vector.
Therefore, if $\M$ has even dimension, then $\mathcal{G}$ is a chirality 
operator in the above sense.
\end{example}

\paragraph{Riemannian chirality boundary conditions.}
Let $\mathcal{G}$ be a chirality operator as before, but we now assume 
$\mathcal{G}$ to commute with $\gamma(\partial_t)$ and to be unitary (with 
respect to $\fiber{\cdot}{\cdot}$).
Consider the projector operator
\[\pi_{\rm 
CHI}:=\frac{1}{2}\left(\mathrm{Id}+\frac{i}{\beta}
\gamma(\n)\gamma(\partial_t)\mathcal
{ G }\right)\,.
\]
Since the Riemannian Clifford multiplication on the spacelike slice $\Sigma_t$ 
is related to the Lorentzian one by 
\begin{equation}\label{eq:RiemCliff}
\gamma_{\Sigma_t}(X)\simeq \frac{\i}{\beta}\gamma(X) \gamma(\partial_t)
\end{equation}
for all $X\in \T\Sigma_t$, we can interpret the range of $\pi_{\rm CHI}$ to be 
a Riemannian 
chirality boundary space.\\
Contrary to the (Lorentzian) chirality boundary condition, the map $\pi_{\rm 
CHI}$ is an orthogonal projection:
 it clearly satisfies $\pi_{\rm 
CHI}^2=\pi_{\rm CHI}$ and, for any $\psi,\varphi\in\mathbb{S}\M$,
\begin{eqnarray*}
\fiber{\pi_{\rm 
CHI}\psi}{\varphi} &=& \frac{1}{2}\fiber{ \psi+\frac{\i}{\beta}
\gamma(\n)\gamma(\partial_t)\mathcal{G}\psi}{ \varphi}\\
&=&\frac{1}{2}\left(\fiber{\psi}{\varphi}-\frac{
1}{\beta}\fiber{\mathcal{G}\psi}{\i
\gamma(\n)\gamma(\partial_t)\varphi } \right)\\
&=&\frac{1}{2}\fiber{\psi}{\varphi+ \frac{
\i}{\beta} 
\gamma(\n)\gamma(\partial_t)\mathcal{G}\varphi } \\
&=& \fiber{\psi}{\pi_{\rm 
CHI}\varphi} \,.
\end{eqnarray*}
Moreover, the range of $\pi_{\rm CHI}$ is the pointwise eigenspace of 
$\frac{\i}{\beta}\gamma(\n)\gamma(\partial_t)\mathcal{G}$ to the eigenvalue $1$ 
and 
$\mathcal{G}$ exchanges that eigenspace with the eigenspace to the eigenvalue 
$-1$ since 
$\{\mathcal{G},\frac{i}{\beta}\gamma(\n)\gamma(\partial_t)\mathcal{G}\}=0$.
Therefore, the range of $\pi_{\rm CHI}$ has dimension 
$2^{\left[\frac{n}{2}\right]-1}$, which is the number of nonnegative 
eigenvalues of the endomorphism $\sigma_{\Dir}(\n^\flat)$.
As another 
consequence of the above computation, we have, for any $\psi\in\mathbb{S}\M$,
\[\gamma(\n)\pi_{\rm 
CHI}\psi=-\frac{\i}{\beta}\gamma(\partial_t)\mathcal{G}\pi_{\rm 
CHI}\psi,\]
where $\gamma(\partial_t)\mathcal{G}$ is Hermitian with respect to  
$\fiber{\cdot}{\cdot}$ since $[\mathcal{G},\partial_t]=0$ by assumption.
Now, for any $\psi\in \gamma(\mathbb{S}\M)$, we obtain
\begin{eqnarray*}
\fiber{\sigma_{\Dir}
(\n^\flat)\pi_{\rm 
CHI}\psi}{\pi_{\rm 
CHI}\psi} &=&\fiber{\gamma(\n)\pi_{\rm 
CHI}\psi}{\pi_{\rm 
CHI}\psi}\\
&=&-\frac{i}{\beta}\fiber{\gamma(\partial_t)\mathcal{G}\pi_{\rm 
CHI}\psi}{\pi_{\rm 
CHI}\psi},
\end{eqnarray*}
and the right-hand side of the last identity is simultaneously real and 
imaginary, 
therefore vanishes.
This proves the Riemannian chirality boundary condition to be admissible.  
Analogous arguments show that the range of $\pi_{\rm 
CHI}:=\frac{1}{2}\left(\mathrm{Id}-\frac{i}{\beta}
\gamma(\n)\gamma(\partial_t)\mathcal
{ G }\right)$ is also an admissible boundary space.

\paragraph{Lorentzian MIT bag boundary conditions.}
Consider the  
so-called MIT bag boundary space, which is defined as the range of 
\[\pi_{\rm 
MIT}:=\frac{1}{2}\left(\mathrm{
Id} - \imath\gamma(\n)\right),\]
where $\gamma(\n)$ is again the Lorentzian Clifford multiplication for the 
outward-pointing unit normal vector along $\bM$.
It is clear 
it is a pointwise linear projection whose range is the pointwise eigenspace of 
$\imath\gamma(\n)$ to the eigenvalue $-1$ and that is exchanged with the other 
eigenspace (to the eigenvalue $1$) by the Clifford multiplication of any 
nonzero vector that is orthogonal to $\n$.
Therefore, the range of $\pi_{\rm MIT}$ has dimension 
$2^{\left[\frac{n}{2}\right]-1}$, which is the number of nonnegative 
eigenvalues of the endomorphism $\sigma_{\Dir}(\n^\flat)$.
Moreover, for any $\psi\in\mathbb{S}\M_{|_{\bM}}$,
\begin{eqnarray*}
\fiber{\sigma_{\Dir}(\n^\flat)\pi_{\rm 
MIT}\psi}{\pi_{\rm 
MIT}\psi}_\beta&=& \fiber{\gamma(\n)\pi_{\rm 
MIT}\psi}{\pi_{\rm 
MIT}\psi}\\
&=&\imath \fiber{\pi_{\rm 
MIT}\psi}{\pi_{\rm 
MIT}\psi},
\end{eqnarray*}
which is simultaneously real and imaginary, therefore vanishes. This proves the 
MIT bag boun\-da\-ry condition to be also admissible.\medskip

Analogous arguments show that the range of $\pi_{\rm 
MIT}:=\frac{1}{2}\left(\mathrm{
Id} + \imath\gamma(\n)\right)$ is also an admissible boundary space.

\paragraph{Riemannian MIT boundary condition.}
We shall now present the Riemannian counterpart of the MIT boundary condition, 
replacing the Clifford multiplication on $\M$ by that along each $\Sigma_t$. 
Motivated by~\eqref{eq:RiemCliff}, consider the operator
\[\pi_{\rm 
MIT}:=\frac{1}{2}\left(\mathrm{
Id}-\frac{1}{\beta}\gamma(\n)\gamma(\partial_t)\right).\]
As the (Lorentzian) MIT boundary condition, it is a projection whose range has 
dimension  $2^{\left[\frac{n}{2}\right]-1}$.
Moreover, since for any $\psi\in\mathbb{S}\M$ it holds
\[\gamma(\n)\pi_{\rm MIT}\psi=\frac{1}{\beta}\gamma(\partial_t)\pi_{\rm 
MIT}\psi,\]
this implies
\begin{eqnarray*}
\fiber{\sigma_{\Dir}(\n^\flat)\pi_{\rm 
MIT}\psi}{\pi_{\rm 
MIT}\psi} &=& \fiber{\gamma(\n)\pi_{\rm 
MIT}\psi}{\pi_{\rm 
MIT}\psi}\\
&=&\frac{1}{\beta}\fiber{\gamma(\partial_t)\pi_{\rm 
MIT}\psi}{\pi_{\rm 
MIT}\psi}\geq 0 \,.
\end{eqnarray*}
This proves the Riemannian  MIT bag boundary space to be also admissible 
 for the forward Cauchy problem.
\color{black} 
Notice that the range of $\pi_{\rm 
MIT}:=\frac{1}{2}\left(\mathrm{
Id}+\frac{1}{\beta}\gamma(\n)\gamma(\partial_t)\right)$ is admissible 
 for the backward Cauchy problem
\color{black}
 since 
we have
\begin{eqnarray*}
\fiber{\sigma_{\Dir}(\n^\flat)\pi_{\rm 
MIT}\psi}{\pi_{\rm 
MIT}\psi} &=& - \frac{1}{\beta}\fiber{\gamma(\partial_t)\pi_{\rm 
MIT}\psi}{\pi_{\rm 
MIT}\psi}\leq 0 \,.
\end{eqnarray*}

\subsection{The geometric wave operator}\label{sec:GWO}

Let $\V$ be a Hermitian vector bundle of finite rank and consider a normally 
hyperbolic operator  $\P\colon\Gamma(\V)\to\Gamma(\V)$ , \ie a $2^{nd}$-order 
linear 
differential operator with principal symbol $\sigma_\P$ defined by 
$$\sigma_\P(\xi)=-g(\xi,\xi)\cdot\mathrm{Id}_\V\,,$$
for every $\xi\in \T^*\M$.
Then $\P$ can be turned into a symmetric hyperbolic system of a first order, 
see e.g.~\cite[Remark 3.7.11]{Ba-lect}.
First, there exists a unique covariant derivative $\nabla$ on $\V$ such that 
$\P=\nabla^*\nabla+c$ for some zero-order term $c$, see \cite[Lemma 
1.5.5]{wave}.
By Theorem~\ref{thm: Sanchez}, the globally hyperbolic manifold $\M$ can be 
written as $(\mathbb{R}\times\Sigma,-\beta^2 dt^2 + h_t)$, where each 
$\{t\}\times\Sigma$ is a smooth spacelike Cauchy hypersurface of $\M$, the 
function $\beta$ is smooth and positive $\mathbb{R}\times\Sigma$ and 
$(h_t)_{t\in\mathbb{R}}$ is a smooth one-parameter-family of Riemannian metrics 
on $\Sigma$.
Then computations show that
$$\nabla^*\nabla=\frac{1}{{\beta^2}}\nabla_{\partial_t}^2+\frac{1}{2{\beta^2}}
\left(\mathrm{tr}_{h_t}(\partial_t 
h_t)-\frac{\partial_t{\beta^2}}{{\beta^2}}\right)\nabla_{\partial_t}
+(\nabla^\Sigma)^*\nabla^{\Sigma}-\frac{1}{2{\beta^2}}\nabla_{\mathrm{grad}_{h_t
}({\beta^2})}^\Sigma,$$
where $\nabla^\Sigma$ is the restricted covariant derivative on $\Sigma$, that 
is, $\nabla_X^\Sigma u:=\nabla_X u$ for all $X\in \T\Sigma$ and 
$u\in\Gamma(\V)$.
Therefore, $\P$ can be written under the form
$$\P=\frac{1}{{\beta^2}}\nabla_{\partial_t}^2+b_0\nabla_{\partial_t}
+(\nabla^\Sigma)^*\nabla^{\Sigma}+\nabla_b^\Sigma+c,$$
where $b_0:=\frac{1}{2{\beta^2}}\left(\mathrm{tr}_{h_t}(\partial_t 
h_t)-\frac{\partial_t{\beta^2}}{{\beta^2}}\right)\in 
C^\infty(\mathbb{R}\times\Sigma,\mathbb{R})$, 
$b:=-\frac{1}{2{\beta^2}}\mathrm{grad}_{h_t}({\beta^2}
)\in\Gamma(\pi_2^*\T\Sigma)$.
This allows us to rewrite the Cauchy problem for $\P$ with boundary condition 
$\Pi_{\oB}\colon\V\oplus(\T^*\Sigma\otimes\V)\oplus\V\to\oB$
\begin{equation}\label{eq:Cauchy P}
 \begin{cases}{}
{\P} u= f   \\
u_{|_{\Sigma_{t_0}}} = h   \\
\nabla_{\partial_t} u_{|_{\Sigma_{t_0}}} = h'   \\
(\nabla_{\partial t}^\V u,\nabla^\Sigma u,u)_{|_{\bM}}\in\oB
\end{cases} 
\end{equation} 
 as a Cauchy problem for $\oS\colon\Gamma(\E) \to \Gamma(\E)$ with boundary 
condition 
$\Pi_\oB\colon\E\to\oB$,  
\begin{equation}\label{eq:Cauchy oS} \begin{cases}{}
{\oS }\Psi:= (A_0 \nabla^\V_{\partial_t} + A_\Sigma \nabla^{\Sigma} + C) \Psi = 
 \f   \\
\Psi|_{\Sigma_{t_0}} = \h   \\
\Psi|_{\bM} \in \oB
\end{cases} 
\end{equation} 
 where $\E$ is the Hermitian vector bundle $\E:=\V\oplus(\T^*\Sigma\otimes 
\V)\oplus \V$, $B\in\Gamma(\text{End}(\E))$  and
$$
\Psi:=\begin{pmatrix}
\nabla^\V_{\partial_t}u\\\nabla^\Sigma u\\ u
\end{pmatrix}, \quad \f := \begin{pmatrix}
f\\ 0 \\ 0
\end{pmatrix}, \quad
A_0:=\begin{pmatrix}
\frac{1}{{\beta^2}}&0&0\\0&1&0\\0&0&1
\end{pmatrix},
$$
$$ A_\Sigma=\begin{pmatrix}
0&-\mathrm{tr}_{h_t}&0\\-1&0&0\\0&0&0
\end{pmatrix}\,\quad C:=\begin{pmatrix}
b_0&b\lrcorner&c\\0&\frac{1}{2}h_t^{-1}\partial_th_t\lrcorner&R_{\partial_t,
\cdot}\\-1&0&0
\end{pmatrix} \,.
$$
The Cauchy problem~\eqref{eq:Cauchy oS} should be read as follows: 
$\nabla_{\partial_t}\nabla^\Sigma u$ is defined
by 
$$\left(\nabla_{\partial_t}\nabla^\Sigma 
u\right)_X:=\nabla_{\partial_t}\nabla_X^\Sigma 
u-\nabla_{(\nabla_{\partial_t}X)^\Sigma}^\Sigma u$$
 for all $X\in\Gamma(\pi_2^*\T\Sigma)$.
The term $\nabla^\Sigma \Psi$ is a section of $(\T^*\Sigma\otimes 
\V)\oplus(\T^*\Sigma\otimes \T^*\Sigma\otimes \V)\oplus (\T^*\Sigma\otimes 
\V)\to\M$, the trace coefficient contracting $\T^*\Sigma\otimes\T^*\Sigma$ of 
course. 
The coefficient $\frac{1}{2}h_t^{-1}\partial_th_t\lrcorner$ is more or less the 
Weingarten map (or shape operator) put into the $\T\Sigma$ slot.
The curvature tensor $R$ is that of $\nabla$ and is by convention given for all 
$X,Y\in \T\M$ by $R_{X,Y}=[\nabla_X,\nabla_Y]-\nabla_{[X,Y]}$.
The only difference with B\"ar's expression for the first-order-operator, apart 
from swapping the first and the second components of $\Psi$, is the vanishing 
of 
the $(2,1)$-coefficient in the zero-order matrix (no coefficient 
$\pi^t(\cdot)$), which plays no role anyway for conditions (S) and (H) since 
those deal with the principal symbol.
\begin{remark}
Notice that, while any solution $u$ of the Cauchy
problem~\eqref{eq:Cauchy P} 
gives a solution $\Psi$ to the Cauchy problem~\eqref{eq:Cauchy oS}, the 
contrary 
does not hold. Indeed, the space of initial data for $\Psi$ is ``too large'' 
and 
some a suitable restriction has to be imposed. For further details we refer to 
\cite[Remark 3.7.11]{Ba-lect}. 
\end{remark}

We summarize the previous observation in the following proposition.

\begin{proposition} Any normally hyperbolic operator $\P$ on a globally 
hyperbolic 
manifold $\M$ with timelike boundary can be reduced to a 
symmetric hyperbolic system $\oS$ of constant characteristic given as 
in~\eqref{eq:Cauchy oS}.
\end{proposition}
\begin{proof}
As in \cite[Remark 3.7.11]{Ba-lect}, Conditions (S) and (H) can be easily 
che\-cked. Moreover, by choosing a Cauchy temporal function with gradient tangent 
to $\bM$, it is easy to see that $\oS$ is of constant characteristic. Indeed, 
since 
\[\sigma_{\oS}(\n^b)=\left(\begin{array}{ccc}
0&-\n^b\lrcorner&0\\-\n^\flat\otimes&0&0\\0&0&0\end{array}\right),
\]
the pointwise kernel of $\sigma_{\oS}(\n^\flat)$ is given by 
\[\ker(\sigma_{\oS}(\n^\flat))=\{0\}\oplus(\n^\flat)^\perp\otimes\V\oplus\V,
\]
which clearly has constant rank.
\end{proof}

\begin{remark} Notice that $\sigma_{\oS}(\n^\flat)$ has pointwise three 
eigenvalues: $0$ of 
multiplicity $nk$, where $n+1=\dim(\M)$ and $k=\mathrm{rk}_{\mathbb{R}}(\V)$, 
$1$ and $-1$, both of the same multiplicity $k$.
Actually, for any 
$\varepsilon\in\{\pm1\}$ 
and for any 
$\Psi=(\Psi_1,\Psi_2,\Psi_3)\in\E$, 
we have 
\begin{eqnarray*}
\sigma_{\oS}(\n^\flat)\Psi=-\varepsilon 
\Psi&\Longleftrightarrow&(-\n^\flat\lrcorner 
\Psi_2,-\n^\flat\otimes \Psi_1,0)=-\varepsilon(\Psi_1,\Psi_2,\Psi_3)\\
&\Longleftrightarrow&\left\{\begin{array}{ll}\n^\flat\lrcorner 
\Psi_2&=\varepsilon 
\Psi_1\\\n^\flat\otimes \Psi_1&=\varepsilon \Psi_2\\ 
\Psi_3&=0\end{array}\right.\\
&\Longleftrightarrow&\left\{\begin{array}{ll}\n^\flat\otimes 
\Psi_1&=\varepsilon 
\Psi_2\\ \Psi_3&=0\end{array}\right.\\
&\Longleftrightarrow& \Psi=(\Psi_1,\varepsilon\n^\flat\otimes \Psi_1,0)\\
&\Longleftrightarrow& \Psi=\left(\left(\mathrm { Id }
\oplus\varepsilon\n^\flat\otimes\right)(\Psi_1),0\right)
\end{eqnarray*}
that is,
\[\ker(\sigma_{\oS}(\n^\flat)+\varepsilon)=\left(\mathrm { Id }
\oplus\varepsilon\n^\flat\otimes\right)(\V)\oplus\{0\}.\]
As a consequence, since $\mathrm { Id }
\oplus\varepsilon\n^\flat\otimes$ is injective, 
$\ker(\sigma_{\oS}(\n^\flat)+\varepsilon)$ has pointwise rank $k$.
In particular,
\[\sum_{\lambda\geq0}\dim(\ker(\sigma_{\oS}(\n^\flat)-\lambda))=(n+1)k.\]
\end{remark}

\begin{definition}\label{def:reduced admissible}
Let $\P$ be a normally hyperbolic operator. We say that ${\oB'}$ is an 
\emph{admissible boundary space} for $\P$ if there exists an admissible 
boundary 
space $\oB$ for $\oS$ such that the Cauchy problems are equivalent.
\end{definition}

Before showing some example of boundary conditions $\Pi_{\oB'}$ for $\P$ which 
reduce to admissible boundary condition $\Pi_\oB$ for $\oS$, let us state and 
prove the main result of this section:

\begin{theorem}
Let $\M$ be a globally hyperbolic manifold with timelike boundary and denote 
with $\oB'$ an admissible boundary space for a normally hyperbolic operator 
$\P\colon\Gamma(\V)\to\Gamma(\V)$.
Then the Cauchy problem for $\P$ is 
well-posed, 
namely for any data $(f,h,h')$ satisfying the compatibility 
condition for any $k\geq0$,  
there exists a unique smooth solution $u 
\in\Gamma(\V)$  to the mixed initial-boundary value problem~\eqref{eq:Cauchy P} 
 which depends continuously on the data $(f,h,h')$.
\end{theorem}

Note that, when we require $(f,h,h')$ to satisfy the compatibility 
condition \eqref{eq:comp cond data} for any $k\geq0$, we mean that the 
corresponding data $(\f,\h)=((f,0,0),(h',\nabla^\Sigma h,h))$ for the 
first-order symmetric hyperbolic system $\oS$ satisfies \eqref{eq:comp cond 
data} for any $k\geq0$.
The proof is a straightforward consequence of Theorem \ref{thm:main2}.\\
\medskip

{\bf\large Examples of admissible boundary conditions}\\

The aim of this section is to test whether particular known boundary conditions 
for normally hyperbolic operators $\P$ are admissible in the sense of 
Definition 
\ref{def:reduced admissible}.

\paragraph{Neumann-like boundary conditions.}

We look at a particular boundary condition, namely the condition
\begin{equation}\label{eq:kindofNeumann}\nabla_\n^\Sigma u=0\end{equation}
along $\bM$.
We could call it the \emph{Neumann-like boundary condition}.
In that case, for the corresponding symmetric hyperbolic systems $\oS$ the 
boundary space $\oB$ coincides with the kernel of the pointwise 
projection 
\[G_{\oB}\colon\E_{|_{\bM}}\longrightarrow\E_{|_{\bM}},\qquad 
G_{\oB}:=\left(\begin{array}{ccc}0&\n\lrcorner&0\\0&0&0\\0&0&0\end{array}
\right).\]
That kernel can be written explicitly down
\[\ker(G_{\oB})=\V\oplus\;(\n^\flat)^\perp\otimes\V\;\oplus \V\]
and direct computations shows that $\dim(\ker(G_{\oB}))=(n+1)k$ pointwise.
Furthermore, for any $\Psi=(\Psi_1,\Psi_2,\Psi_3)\in\ker(G_{\oB})$,
\begin{eqnarray*}
\langle\sigma_{\oS}(\n^\flat)\Psi,\Psi\rangle&=&\langle(-\n\lrcorner 
\Psi_2,-\n^\flat\otimes 
\Psi_1,0),(\Psi_1,\Psi_2,\Psi_3)\rangle\\
&=&-2\Re e(\langle\n\lrcorner \Psi_2,\Psi_1\rangle)=0
\end{eqnarray*}
where we used $\n\lrcorner \Psi_2= \nabla_\n^\Sigma u=0$ since 
$\Psi\in\ker(G_{\oB})$.
This proves (\ref{eq:kindofNeumann}) to be admissible in the sense of 
Definition~\ref{def:reduced admissible}.

\paragraph{Transparent boundary conditions.}

The transparent boundary condition is defined as
\begin{equation}\label{eq:semitransparentbc}
\nabla_\n^\Sigma u=-b\nabla_{\partial_t}u
\end{equation}
along $\bM$ for some real parameter $b$, see e.g. \cite[Eq. 
(1)]{EngquistMajda77}.
In that case, the bundle $\oB$ coincides with the kernel of the pointwise 
projection 
\[G_{\oB}\colon\E_{|_{\bM}}\longrightarrow\E_{|_{\bM}},\qquad 
G_{\oB}:=\left(\begin{array}{ccc}b&\n\lrcorner&0\\0&0&0\\0&0&0\end{array}
\right),\]
that is, 
\begin{eqnarray*}
\oB&=&\ker(G_{\oB})\\
&=&\left\{\left(-\frac{1}{b}\n\lrcorner 
X_2,X_2,X_3\right)\,|\,(X_2,X_3)\in 
\T^*\Sigma\otimes\V\;\oplus\;\V\right\}\\
&=&\left(-\frac{1}{b}
\n\lrcorner\cdot\oplus\,\mathrm{Id}\right)(\T^*\Sigma\otimes V)\;\oplus\;\V.
\end{eqnarray*}

In particular, $\mathrm{rk}_{\mathbb{R}}(\oB)=(n+1)k$, as required.
Moreover, for any $X=(X_1,X_2,X_3)\in\oB$,
\[\langle\sigma_{\oS}(\n^b)X,X\rangle=-2\Re 
e(\langle\n\lrcorner X_2,X_1\rangle)=\frac{2}{b}|\n\lrcorner X_2|^2,\]
which is nonnegative as soon as $b\geq0$.
This shows (\ref{eq:semitransparentbc}) to be admissible 
for the forward Cauchy problem when $b\geq0$, 
 while if admissible for the backward Cauchy problem if $b\leq 0$.
\color{black}\\
\medskip

{\bf\large An example of a non-admissible boundary condition.}

\paragraph{Robin boundary condition for differential forms.}

In the particular situation where $\V=\Lambda^p\T^*\M$ is the bundle of 
differential 
forms on $\M$ for some $p\in\{0,1,\ldots,n+1\}$, there is  another 
boundary condition called the Robin boundary condition.
It is defined, for any $p$-form $\omega$ on $\M$ by 
\[\left\{\begin{array}{ll}\iota^*(\n\lrcorner 
 d\omega)&=\tau\iota^*\omega\\\iota^*(\n\lrcorner\omega)&=0\end{array}\right.,\]
where $\tau$ is a real parameter.
Here $d$ denotes the exterior differential as usual and 
$\iota\colon\bM\longrightarrow\M$ is the inclusion map.
The case where $\tau=0$ is considered (at least by some geometric analysts) as 
the ``standard'' generalization of the Neumann boundary condition for forms; it 
is usually called ``absolute boundary condition'' in the literature (there are 
also relative ones).
For Robin boundary conditions -- we let $\tau$ be any real parameter for the 
time being, so this includes the absolute boundary condition -- the bundle 
$\oB$ 
is the kernel of the pointwise projection
\[G_{\oB}:=\left(\begin{array}{ccc}
-dt\wedge(\n\lrcorner\cdot)&\iota^*(\n\lrcorner\cdot)-\sum_{j=2}
^ne_j^*\wedge\iota^*(\n\lrcorner 
e_j\lrcorner\cdot)&-\tau\iota^*\\0&0&0\\0&0&\iota^*(\n\lrcorner\cdot)
\end{array}\right),\]
where $(e_j)_{2\leq j\leq n}$ denotes any pointwise o.n.b. of 
$T(\bM\cap\Sigma)$.
Next we make $\oB$ a bit more precise.\\

It is already clear that 
\[\ker\left(\iota^*(\n\lrcorner\cdot)\right)=\left\{\omega\in\V\,|\,
\iota^*\omega=\omega\right\}=\Lambda^p\T^*\bM,\]
whose pointwise rank is $\left(\begin{array}{c}n\\ p\end{array}\right)$.
To see what condition the first line in the above matrix $G_{\oB}$ gives, we 
split any $\omega\in \V$ as follows:
\begin{eqnarray*}
\omega&=&\n^*\wedge\omega^{(1)}+\omega^T\\
&=&\n^*\wedge 
dt\wedge(\omega^{(1)})^{(t)}+n^*\wedge(\omega^{(1)})^{\bM\cap\Sigma}+\omega^T\\
&=&\n^*\wedge 
dt\wedge(\omega^{(1)})^{(t)}+n^*\wedge(\omega^{(1)})^{\bM\cap\Sigma}
+dt\wedge(\omega^T)^{(t)}+(\omega^T)^{\bM\cap\Sigma},
\end{eqnarray*}
where $(\omega^{(1)})^{(t)}\in\Lambda^{p-2}\T^*(\bM\cap\Sigma)$, 
$(\omega^{(1)})^{\bM\cap\Sigma},(\omega^T)^{(1)}\in\Lambda^{p-1}
\T^*(\bM\cap\Sigma)$ and $(\omega^T)^{\bM\cap\Sigma}\in\Lambda^p 
\T^*(\bM\cap\Sigma)$.
For any $X=(X_1,X_2,X_3)\in\oB$, we write 
$X_2=\sum_{j=1}^ne_j^*\otimes\omega_j$, where 
$(e_1=\nu,e_2,\ldots,e_n)$ is a pointwise o.n.b. of 
$\T\Sigma$ and $\omega_j\in\V$ for all $j\geq1$.
Then, setting 
$\Omega:=-dt\wedge(\n\lrcorner X_1)+\iota^*(X_2(\n))-\sum_{j=2}
^ne_j^*\wedge\iota^*(\n\lrcorner 
 X_2(e_j))-\tau\iota^*X_3$, we compute and use $\iota^* X_3=X_3$ as well as $\iota^*(\n\lrcorner\omega_j)=\n\lrcorner\omega_j$:
\begin{eqnarray*}
\Omega&=&-dt\wedge(\n\lrcorner X_1)+\iota^*\omega_1-\sum_{j=2}
^ne_j^*\wedge(\n\lrcorner 
 \omega_j)-\tau X_3\\
&=&-dt\wedge X_1^{(1)}+\omega_1^T-\sum_{j=2}^n e_j^*\wedge\omega_j^{(1)}-\tau 
X_3\\
&=&-dt\wedge\left(dt\wedge(X_1^{(1)})^{(t)}+(X_1^{(1)})^{\bM\cap\Sigma}
\right)+dt\wedge(\omega_1^T)^{(t)}+(\omega_1^T)^{\bM\cap\Sigma}\\
&&-\sum_{j=2}^ne_j^*\wedge\left(dt\wedge(\omega_j^{(1)})^{(t)}+(\omega_j^{(1)})^
{\bM\cap\Sigma}\right)-\tau dt\wedge(X_3)^{(t)}-\tau X_3^{\bM\cap\Sigma}\\
&=&-dt\wedge(X_1^{(1)})^{\bM\cap\Sigma}+dt\wedge(\omega_1^T)^{(t)}+(\omega_1^T)^
{\bM\cap\Sigma}\\
&&+dt\wedge\left(\sum_{j=2}^ne_j^*\wedge(\omega_j^{(1)})^{(t)}\right)-\sum_{j=2}
^n e_j^*\wedge(\omega_j^ { (1) } )^
{\bM\cap\Sigma}-\tau dt\wedge(X_3)^{(t)}\\
&&-\tau X_3^{\bM\cap\Sigma}\\
&=&dt\wedge\left((\omega_1^T)^{(t)}-(X_1^{(1)})^{\bM\cap\Sigma}+\sum_{j=2}
^ne_j^*\wedge(\omega_j^{(1)})^{(t)}-\tau(X_3)^{(t)}\right)\\
&&+(\omega_1^T)^
{\bM\cap\Sigma}-\sum_{j=2}
^n e_j^*\wedge(\omega_j^ { (1) } )^
{\bM\cap\Sigma}-\tau X_3^{\bM\cap\Sigma}.
\end{eqnarray*}
Therefore, $\Omega=0$ if and only if 
\begin{equation}\label{eq:condOmega=0}
\left\{\begin{array}{ll} 
(\omega_1^T)^{(t)}-(X_1^{(1)})^{\bM\cap\Sigma}+\sum_{j=2}
^ne_j^*\wedge(\omega_j^{(1)})^{(t)}-\tau(X_3)^{(t)}&=0\\
&\\
(\omega_1^T)^
{\bM\cap\Sigma}-\sum_{j=2}
^n e_j^*\wedge(\omega_j^ { (1) } )^
{\bM\cap\Sigma}-\tau X_3^{\bM\cap\Sigma}&=0
\end{array}
\right.
\end{equation}
This first shows that both components
$(\omega_1^T)^{(t)}$ and $(\omega_1^T)^
{\bM\cap\Sigma}$ (so actually $\omega_1^T$) depend linearly on other components 
of $X$, however all other components of $X$ can be chosen arbitrarily.
Thus, the space of all $X_1$-components has dimension 
$\left(\begin{array}{c}n+1\\ p\end{array}\right)$; the space of all 
$X_2$-components has dimension $\left(\begin{array}{c}n-1\\ 
p-2\end{array}\right)+\left(\begin{array}{c}n-1\\ 
p-1\end{array}\right)+(n-1)\left(\begin{array}{c}n+1\\ p\end{array}\right)$, 
the first two terms corresponding to the components 
$(\omega_1^{(1)})^{(t)}\in\Lambda^{p-2}\T^*(\bM\cap\Sigma)$ and 
$(\omega_1^{(1)})^{\bM\cap\Sigma}\in\Lambda^{p-1}\T^*(\bM\cap\Sigma)$ 
respectively (actually both just correspond to 
$\omega_1^{(1)}$ lying pointwise in $\Lambda^{p-1}\T^*\bM$) and the last one to the components 
$\omega_2,\ldots,\omega_n\in\Lambda^p\T^*\M$; and the space of all 
$X_3$-components has dimension $\left(\begin{array}{c}n\\ 
p\end{array}\right)$ since $X_3\in\Lambda^p\T^*\bM$.
All in all, $\oB$ has rank
\[\left(\begin{array}{c}n+1\\ p\end{array}\right)+\left(\begin{array}{c}n\\ 
p-1\end{array}\right)+(n-1)\left(\begin{array}{c}n+1\\ 
p\end{array}\right)+\left(\begin{array}{c}n\\ 
p\end{array}\right)=(n+1)\left(\begin{array}{c}n+1\\ p\end{array}\right),\]
which is exactly the rank of $\ker(\sigma_{\oS}(\n^b))$.\\
Now we look at the sign of the quadratic form 
$X\mapsto\langle\sigma_{\oS}(\n^b)X,X\rangle$ on $\oB$.
Given $X=(X_1,X_2,X_3)\in\oB$, we can express as above 
\begin{eqnarray*}
\langle\sigma_{\oS}(\n^b)X,X\rangle&=&-2\Re e(\langle\n\lrcorner 
X_2,X_1\rangle)\\
&=&-2\Re e(\langle\omega_1,X_1\rangle)\\
&=&-2\Re e(\langle\omega_1^{(1)},X_1^{(1)}\rangle)-2\Re 
e(\langle\omega_1^T,X_1^T\rangle)\\
&=&-2\Re e(\langle(\omega_1^{(1)})^{(t)},(X_1^{(1)})^{(t)}\rangle)\\
&&-2\Re 
e(\langle(\omega_1^{(1)})^{\bM\cap\Sigma},(X_1^{(1)})^{\bM\cap\Sigma}
\rangle)-2\Re 
e(\langle\omega_1^T,X_1^T\rangle).
\end{eqnarray*}
But, as we mentioned above, the components $(X_1^{(1)})^{(t)}$ of $X_1$ and 
$(\omega_1^{(1)})^{(t)}$ of $\omega_1$ (which is itself a component of $X_2$) 
can be 
chosen arbitrarily.
Furthermore, none of the components 
$(\omega_1^{(1)})^{\bM\cap\Sigma},(X_1^{(1)})^{\bM\cap\Sigma},\omega_1^T,X_1^T$ 
depend on $((X_1^{(1)})^{(t)},(\omega_1^{(1)})^{(t)})$.
Therefore whatever the 
value of the real number $\Re 
e(\langle(\omega_1^{(1)})^{\bM\cap\Sigma},(X_1^{(1)})^{\bM\cap\Sigma}
\rangle)+\Re 
e(\langle\omega_1^T,X_1^T\rangle)$ is, and provided $p\geq1$, we can always 
choose $(X_1^{(1)})^{(t)}$ 
and $(\omega_1^{(1)})^{(t)}$ such that 
\[\langle\sigma_{\oS}(\n^b)X,X\rangle<0.\]
In case $p=0$, $\omega=\omega^{\bM\cap\Sigma}$ (all other 
components vanish) and hence (\ref{eq:condOmega=0}) is equivalent to 
$\omega_1=\tau X_3$.
Therefore $\langle\sigma_{\oS}(\n^b)X,X\rangle=-2\Re 
e(\omega_1\overline{X_1})$ which vanishes if $\tau=0$ (because then 
$\omega_1=0$) and whose sign can be arbitrary if $\tau\neq0$, as we have 
already seen for condition (\ref{eq:kindofNeumann}) which coincides with the 
Robin boundary condition in that case.
This proves the Robin boundary condition to be non-admissible unless $p=0$ and 
$\tau=0$.\\

If $p\geq1$, there is actually an eigenvector of $\sigma_{\oS}(\n^b)$ 
associated to the 
eigenvalue $-1$ that lies in $\oB$: choose $X=(X_1,\n^b\otimes X_1,0)$ with 
$X_1=\n^*\wedge dt\wedge(X_1^{(1)})^{(t)}+(\n^*+dt)\wedge(X_1^T)^{(t)}$, then 
$X\in\oB\cap\ker(\sigma_{\oS}(\n)+1)$ and therefore 
\[\langle\sigma_{\oS}(\n^b)X,X\rangle=-|X|^2<0\]
as soon as $(X_1^{(1)})^{(t)}$ or $(X_1^T)^{(1)}$ is nonzero.
\section{Examples of symmetric positive systems}\label{sec:Ex SPS}

\subsection{Klein-Gordon operator}\label{sec:Klein-Gordon}

Let $\nabla$ be a covariant derivative on a Hermitian vector bundle $\V$ of 
finite rank $k$ over a globally hyperbolic manifold $\M$  with timelike 
boundary.
The Klein-Gordon operator reads as
$\P=\nabla^*\nabla +m^2$, where $m$ is the mass of the scalar field.
It is by defintion a 
normally hyperbolic operator and hence its Cauchy problem can 
be written as in ~\eqref{eq:Cauchy P}.
Unlike in Section~\ref{sec:GWO}, we 
can rewrite the Cauchy problem for $\P$ in terms of the Cauchy problem for 
the symmetric positive system $\oS\colon\Gamma(\E) \to \Gamma(\E)$, namely
\[ \begin{cases}{}
{\oS }\Psi=  \f   \\
\Psi_{|_{\Sigma_{t_0}}} = \h   \\
\G_\oB \Psi_{|_{\bM}} =0
\end{cases} 
\]
 where $\G_\oB$ is a boundary condition, $\E$ is the Hermitian vector bundle 
$\E:=\V\oplus\T^*\M\otimes \V$
and
$$\oS= \begin{pmatrix}
0 & -\tr \\
-1 & 0 
\end{pmatrix} \nabla + \begin{pmatrix}
 m^2 & 0 \\ 0 & 1 
\end{pmatrix} \qquad \text{with} \qquad \Psi=\begin{pmatrix}
u \\ \nabla u
\end{pmatrix} \,, \;\; F=\begin{pmatrix}
\f \\ 0
\end{pmatrix}\,.$$
As in Section~\ref{sec:GWO}, some restriction on the space of initial data for 
$\Psi$ has to be imposed in order to obtain a correspondence between the Cauchy 
problems.

It is not difficult to check that the principal symbol  $\sigma_\oS(\xi)$ is 
Hermitian for every $\xi\in \T^*\M$. 
Moreover, since the trace can be seen as a contraction of tensors, the 
principal 
symbol is parallel, \ie $\nabla\sigma_\oS=0$ and we get \[\Re 
e(\oS+\oS^\dagger)= 
\begin{pmatrix}
 2m^2 & 0 \\ 0 & 2
\end{pmatrix} \,.\]
 Hence, $\oS$ is a nowhere characteristic symmetric positive 
system.
 Indeed, 
since 
\[\sigma_{\oS}(\n^b)=\left(\begin{array}{cc}
0&-\n^b\lrcorner\\-\n^\flat\otimes&0\end{array}\right),
\]
the pointwise kernel of $\sigma_{\oS}(\n^\flat)$ is given by 

$$\ker \sigma_\oS(\n^\flat)=\{0\}\oplus \n^\perp \otimes V\,,$$
where $\n$ denotes again the normal vector to $\bM$. 
Notice that $\sigma_{\oS}(\n^\flat)$ has pointwise two further 
eigenvalues $1$ and $-1$ both with multiplicity $k$.

The net advantage of this reduction is that the Robin boundary conditions for 
$\P$ can be rewritten as an admissible boundary condition for $\oS$.
Note that, if $P=\Dir^2$ is the squared Dirac operator on $\M$ 
assumed to be spin, then the Schr\"odinger-Lichnerowicz formula states that 
$P=\nabla^*\nabla+\frac{\mathrm{Scal}}{4}$, where $\mathrm{Scal}$ is the scalar 
curvature of $(\M,g)$.
If $\mathrm{Scal}$ is bounded below by a positive constant on $\M$, then by 
analogous arguments as those described above $P$ can be turned into a 
first-order symmetric positive system and therefore the analysis we have 
developed for that category of operators can also be applied.
This is particularly interesting when looking at certain boundary 
conditions.\\
\medskip 

{\bf\large Examples of admissible boundary conditions}

\paragraph{Robin boundary condition.}
The Robin boundary conditions for the Klein-Gordon operator reads as
$$ a \nabla_\n u - b u =0$$
for some real constant parameters $a,b$.
In that case, the bundle $\oB$ coincides with the kernel of the pointwise 
projection 
\[ G_{\oB}:=\begin{pmatrix}
-b & a \n \lrcorner \\
0 & 0
\end{pmatrix}\]
and it has rank $k$ has required.
For any $\Psi=(\Psi_0,\Psi_1)\in\ker G_\oB$ it holds $a \n\lrcorner \Psi_1=b 
\Psi_0$. 
If $a,b\geq 0$ or $a,b\leq 0$, we get
$$\fiber{\sigma_\oS(\n^\flat)\Psi}{\Psi}= 2 \Re e \fiber{\n \lrcorner \Psi_1}{ 
\Psi_0}\geq 
0\,,$$
 showing that, if $ab\geq0$, then the Robin boundary conditions are admissible 
for the forward Cauchy problem\color{black}.
 Note that those Robin boundary conditions should not be confused with the ones 
arising in elliptic systems such as \cite[Theorem 6.31]{GilbargTrudinger77}, 
where $ab<0$ has to be assumed.

\subsection{Diffusion-reaction system}\label{sec:reactiondiffusion}
As for Section~\ref{sec:Klein-Gordon}, let $\nabla$  be a metric connection on 
an Hermitian vector bundle $\V$ of rank $k$.
Consider the diffusion-reaction operator 
$$\P:=\nabla_{\partial_t} - \tr(\nabla^\Sigma\nabla^\Sigma) + c $$
where $c$ is a zero order term, dubbed \emph{linear reaction term}.
The
notation here is the same as in Section~\ref{sec:GWO}. These systems are used 
to model a wide range of phenomena in physics, biology, social sciences, see 
e.g. (\cite{manu,manu1,manu2,manu3}) and the prototype example of a
diffusion-reaction system is the heat equation, where $c$ is set to zero.
Note that this is not the usual way to handle such evolution equations but we 
emphasize those equations fit into our framework.\\
Let us rewrite the Cauchy problem for the diffusion-reaction 
operator in terms of the Cauchy problem for the first order 
symmetric system defined by
\[\oS= \begin{pmatrix}
1 & 0 \\
0 & 0 
\end{pmatrix} \nabla_{\partial_t} + \begin{pmatrix}
0 & -\tr \\
- 1 & 0 
\end{pmatrix} \nabla^\Sigma + \begin{pmatrix}
 c & 0 \\ 0 & 1 
\end{pmatrix}\,.
\]
This equivalence can be obtained by setting 
$\Psi=\begin{pmatrix}
u \\ \nabla^\Sigma u
\end{pmatrix}$.
Differently from the case of the Klein-Gordon operator treated in 
Section~\ref{sec:Klein-Gordon}, $\oS$ is not a symmetric positive system if $c$ 
not positive definite. However, we can use a similar trick as in 
Lemma~\ref{lem:sps in M_T}, to obtain the Property (P) of 
Definition~\ref{def:symm syst}. To this end, let us assume $c$ to be uniformely 
bounded from below and chose a positive $\lambda$ such that $\lambda - c > 0$. 
Then the operator $\oK\colon \Gamma(\E_{|_{\M_\T}}) \to 
\Gamma(\E_{|_{\M_\T}})$ defined by
\[\oK:=  \oS + \lambda \begin{pmatrix}
1 & 0 \\
0 & 0 
\end{pmatrix}\,.
\]
$\oK$ is clearly a symmetric system. Futhermore, its Cauchy problem is 
equivalent to the one of $\oS$, namely
\begin{equation*}
\begin{cases}{}
\oK\tilde\Psi=\tilde \f  \\
\tilde\Psi|_{\Sigma_0} = \tilde\h \\
\tilde\Psi \in \oB
\end{cases} 
\qquad\Longleftrightarrow\qquad
\begin{cases}{}
\oS \Psi=\f   \\
\Psi|_{\Sigma_0} ={\h} \\
 \Psi \in \oB,
\end{cases}
\end{equation*}
where $\tilde\f=e^{-\lambda t}\f$, $\tilde\h=\h$ and $\tilde\Psi=e^{-\lambda 
t}\Psi$. Indeed, we have, for every 
$\phi\in\Gamma(E)$ and for every $t\in\RR$,
\begin{eqnarray*}
\oK(e^{-\lambda t}\phi)&=&\left(\oS+\lambda \begin{pmatrix}
1 & 0 \\
0 & 0 
\end{pmatrix}\right)(e^{-\lambda 
t}\phi)\\
&=&-\lambda e^{-\lambda t}\begin{pmatrix}
1 & 0 \\
0 & 0 
\end{pmatrix}\phi+e^{-\lambda 
t}\left(\oS+\lambda \begin{pmatrix}
1 & 0 \\
0 & 0 
\end{pmatrix}\right)\phi\\
&=&e^{-\lambda t}\oS\phi,
\end{eqnarray*}
Since $\lambda - c>0$ by assumption and the principal symbol is parallel, then 
a straighforward computations shows that $\oK$ is a positive symmetric system. 
Of course, a restriction on the class of initial data for $\oS$, and 
consequently for $\oK$ has to be imposed to get an equivalence between the 
Cauchy problem for $\oS$ and the one for $\P$.\\
\newpage

{\bf\large Examples of admissible boundary conditions}

\paragraph{Robin boundary condition.}

The Robin boundary conditions for the diffusion-reaction system reads as
$$ a \nabla_\n u - b u =0$$
for some real parameters $a,b$.
In that case, the bundle $\oB$ coincides with the kernel of the pointwise 
projection 
\[ G_{\oB}:=\begin{pmatrix}
-b & a \n \lrcorner \\
0 & 0
\end{pmatrix}\]
and it has rank $k$ has required.
For any $\Psi=(\Psi_0,\Psi_1)\in\ker G_\oB$ it holds $a \n\lrcorner \Psi_1=b
\Psi_0$. 
As in Section \ref{sec:Klein-Gordon}, if 
$ab\geq0$, we get
$$\fiber{\sigma_\oS(\n^\flat)\Psi}{\Psi}= 2 \Re e \fiber{\n \lrcorner \Psi_1}{ 
\Psi_0}\geq 
0\,,$$
 showing that the Robin boundary condition are admissible
 for the forward Cauchy problem
 under that 
assumption.

\end{document}